\documentclass[11pt,reqno]{amsart}

\usepackage[margin=3cm]{geometry}
\usepackage{lineno}
\usepackage{stmaryrd}

\usepackage[usenames,dvipsnames,svgnames,table]{xcolor}
\usepackage{amssymb}
\usepackage{amsfonts}
\usepackage{amsmath,amscd}
\usepackage{mathrsfs}
\usepackage{verbatim}
\usepackage{enumitem}

\usepackage{mathtools}
\usepackage{hyperref}
\usepackage{fourier}

\makeatletter \@namedef{subjclassname@2020}{\textup{2020} Mathematics
  Subject Classification} \makeatother

\newcommand{\cred}{\color{Red}}
\newcommand{\cgreen}{\color{Green}}
\newcommand{\cblue}{\color{Blue}}
\newcommand{\corange}{\color{Orange}}
\newcommand{\cteal}{\color{Teal}}
\newcommand{\colive}{\color{Olive}}
\newcommand{\cviolet}{\color{Violet}}

\newtheorem{theorem}{Theorem}[section]
\newtheorem{lemma}[theorem]{Lemma}     
\newtheorem{corollary}[theorem]{Corollary}
\newtheorem{proposition}[theorem]{Proposition}
\newtheorem*{theorem*}{Theorem}
\newtheorem*{corollary*}{Corollary}

\theoremstyle{definition}

\newtheorem{remark}[theorem]{Remark}

\newtheorem{question}[theorem]{Question}

\newtheorem{example}[theorem]{Example}
\newtheorem{notation}[theorem]{Notation}

\newtheorem*{remark*}{Closing Remark}

\newcommand{\se}{\langle}
\newcommand{\sd}{\rangle}
\newcommand{\fac}{\operatorname{F}}
\newcommand{\setl}{\operatorname{L}}
\newcommand{\ulf}{\operatorname{ULF}}
\newcommand{\frob}{\operatorname{Frob}}
\newcommand{\rem}{\varepsilon}
\newcommand{\mbn}{\mathbb{N}}
\newcommand{\mbz}{\mathbb{Z}}
\newcommand{\mbk}{\mathbb{K}}
\newcommand{\la}{\mathscr{L}_a}
\newcommand{\msl}{\mathscr{L}}
\newcommand{\mcb}{\mathcal{B}}

\newcommand{\mon}{\text{m}}
\newcommand{\nums}{S}
\newcommand{\numGen}{\mathscr{S}}
\newcommand{\length}{\mbox{${\rm length}$}}
\newcommand{\card}{\mbox{${\rm card}$}}
\newcommand{\betti}{\operatorname{Betti}}
\newcommand{\bbetti}{\operatorname{BBetti}}
\newcommand{\ubetti}{\operatorname{UBetti}}

\makeatletter
\def\@bignumber#1#2{%
  \ifx#2\end
    #1\let\next\@gobble
  \else
    #1\hspace{0pt plus 1pt}\let\next\@bignumber
  \fi
  \next#2}
\newcommand{\bignumber}[1]{\@bignumber#1\end}
\makeatother

\title[Semigroups with three consecutive generators]{Elements with
  unique length factorization of a numerical semigroup generated by
  three consecutive numbers}

\author[P. A. Garc\'{\i}a-S\'anchez]{Pedro A. Garc\'{\i}a-S\'anchez}
\thanks{The first author is partially supported by grants
  ProyExcel\_00868 and FQM--343, both funded by the Junta de
  Andaluc\'ia. He also acknowledges financial support from grant
  PID2022-138906NB-C21, funded by
  MICIU/AEI/\bignumber{10.13039}/\bignumber{501100011033} and by ERDF
  ``A way of making Europe''; and from the Spanish Ministry of Science
  and Innovation (MICINN), through the ``Severo Ochoa and Mar\'{\i}a de
  Maeztu Programme for Centres and Unities of Excellence''
  (CEX2020-001105-M).}  \address{Departamento de \'Algebra e IMAG,
  Universidad de Granada. Av. Fuentenueva, S/N. E18071
  Granada. Spain.}  \email{pedro@ugr.es} \author[L. Gonz\'alez]{Laura
  Gonz\'alez} \author[F. Planas-Vilanova]{Francesc Planas-Vilanova}
\address{Departament de Matem\`atiques, Universitat Polit\`ecnica de
  Catalunya. Diagonal 647, ETSEIB, E-08028 Barcelona. Spain.}
\email{laura.gonzalez.hernandez@upc.edu}
\email{francesc.planas@upc.edu} \thanks{The second and third authors
  are partially supported by Grant PID2019-103849GB-I00 funded by
  MICIU/AEI/10.13039/501100011033 and AGAUR project 2021 SGR 00603}

\date{\today}


\begin{document}

\begin{abstract}
Let $\nums$ be the numerical semigroup generated by three consecutive
numbers $a,a+1,a+2$, where $a\in\mbn$, $a\geq 3$. We describe the
elements of $\nums$ whose factorizations have all the same length, as
well as the set of factorizations of each of these elements. We give
natural partitions of this subset of $\nums$ in terms of the length
and the denumerant. By using Ap\'ery sets and Betti elements we are
able to extend some results, first obtained by elementary means.
\end{abstract}

\maketitle 

\section{Introduction}\label{sec-intro}

Let $\nums=\se a,a+1,a+2\sd$ be the numerical semigroup generated by
$a,a+1,a+2$, where $a\in\mbn$, $a\geq 3$, $\mbn$ being the set of
non-negative integers. The \emph{Frobenius number} of $\nums$, i.e.,
the largest integer that does not belong to $\nums$, is
$\frob(\nums)=\lfloor a/2\rfloor a-1$ (see,
e.g.,\cite[Corollary~5]{gr}, \cite[Theorem~3.3.1]{ramirez} or
\cite{so}). If $r\in\mbn$, let
\begin{eqnarray*}
&&\fac(r,\nums)=\{\alpha=(\alpha_1,\alpha_2,\alpha_3)\in\mbn^3 : 
\alpha_1 a+\alpha_2(a+1)+\alpha_3(a+2)=r\}
\end{eqnarray*}
stand for the \emph{set of factorizations of} $r$. Note that if
$r\not\in\nums$, $\fac(r,\nums)=\emptyset$.  The following vector
space determines $\fac(r,\nums)$: let $x,y,z$ be variables over a
field $\mbk$ and $W_r= \se \mon^{\alpha} : \alpha\in\fac(r,\nums)\sd$
be the $\mbk$-vector space spanned by the monomials
$\mon^{\alpha}:=x^{\alpha_1}y^{\alpha_2}z^{\alpha_3}$, with
$\alpha\in\fac(r,\nums)$. If $r\not\in\nums$, we just define
$W_r=\{0\}$.  Let $\delta_r=\dim W_r$ be the dimension of $W_r$ as a
$\mbk$-vector space. Note that $\delta_r$ is equal to the cardinality
of $\fac(r,\nums)$, which is known as the {\em denumerant} of $r$. If
$r\in\nums$, then $\delta_r\geq 1$. For any $\alpha\in\fac(r,\nums)$,
the \emph{length} of $\alpha$ is defined as
$\length(\alpha)=|\alpha|=\alpha_1+\alpha_2+\alpha_3$.  Let
\begin{eqnarray*}
\setl(r,\nums)=\{|\alpha| :  \alpha\in\fac(r,\nums)\}
\phantom{+}\mbox{and}\phantom{+}
\ulf(\nums)=\{r\in\nums :  \card(\setl(r,\nums))=1\}
\end{eqnarray*}
be the set of lengths of $r\in\nums$ and the set of elements of
$\nums$ with a \emph{unique length factorization}, respectively. Given
$\ell,d\in\mbn$, $d\geq 1$, let
\begin{eqnarray*}
\nums^{\ell}=\{r\in\nums : \setl(r,\nums)=\{\ell\}\}\phantom{+}
\mbox{and}\phantom{+} \nums_d=\{r\in\nums : \card(\fac(r,\nums))=d\}.
\end{eqnarray*}
For $a,b\in\mbz$, set $\llbracket a,b\rrbracket=\{r\in\mbz : a\leq
r\leq b\}$ and similarly for $\llbracket a,b)\!)$, $(\!(a,b\rrbracket$
and $(\!(a,b)\!)$.  For $x$ a real number, set $\lfloor
x\rfloor=\max\{z\in \mathbb{Z} : z\le x\}$ and $\lceil
x\rceil=\min\{z\in \mathbb{Z} : x\le z\}$ (the floor and ceiling
functions of $x$, respectively).  Let $\ell_r=\lfloor r/a\rfloor$.

We summarise now the present work. In the first part, the proofs are
simple and there is no need of previous knowledge on numerical
semigroups. We begin by proving that there exists $\la\in\nums$, such
that $\nums\cap \llbracket 0,\la)\!)\subseteq \ulf(\nums)$ and such
that $\la\not\in\ulf(\nums)$. Moreover, $\setl(r,\nums)=\{\ell_r\}$,
for each $r\in\nums\cap \llbracket 0,\la)\!)$
(Theorem~\ref{theolength}).  Given $r\in\mbn$, we introduce an integer
vector $\phi_r\in\mbz^3$ and study the membership problem of $r$, in
terms of $\ell_r$ and $\phi_r\in\mbz^3$. When $r\in\nums\cap
\llbracket 0,\la)\!)$, then $\phi_r\in\mbn^3$ and $\phi_r$ becomes a
``seed'' factorization which allows us to completely describe all the
factorizations of $r$ (Theorem~\ref{theo-length}). As a corollary, we
obtain a partition of the subset $\nums\cap \llbracket
0,(a+2)L\rrbracket$ in terms of $\nums^{\ell}$
(Corollary~\ref{cor-length}), where $L=\lfloor
(a-1)/2\rfloor$. Subsequently, we analise the membership problem of
$r$ in terms of the denumerant $\delta_r$ and two other invariants,
$\iota_r$ and $c_r$, attached to the seed vector $\phi_r$ of $r$
(Theorem~\ref{theo-dim}). As a consequence, we get another partition
of $\nums\cap \llbracket 0,(a+2)L\rrbracket$, now in terms of
$\nums_d$ (Corollary~\ref{cor-dim}).

In the second part of the paper, by using Ap\'ery sets and Betti
elements (see, e.g., \cite{rg}, \cite{ns-ap}), we are able to achieve
some more general results. Let $\numGen$ be any numerical
semigroup. Denote by $\betti(\numGen)$ the set of Betti elements of
$\numGen$ and set
$\bbetti(\numGen)=\betti(\numGen)\bigcap\ulf(\numGen)$ and
$\ubetti(\numGen)=\betti(\numGen)\setminus \bbetti(\numGen)$ (B and U,
standing for \emph{balanced} and \emph{unbalanced}).  Then, in
Theorem~\ref{th-appendix}, we show that
$\ulf(\numGen)=\operatorname{Ap}(\numGen,\ubetti(\numGen))$ and, in
Corollary~\ref{cor-minbetti}, we show that $\ulf(\numGen)$ is a finite
set, $\numGen\cap \llbracket 0,b)\!)\subseteq \ulf(\numGen)$, where
$b=\min(\ubetti(\numGen))$, and $b\not\in\ulf(\numGen)$.  Again with
$\numGen=\nums=\se a,a+1,a+2\sd$, one describes the whole set
$\ulf(\nums)$ (see Lemmas~\ref{lem:ap-a-even}
and~\ref{lem:ap-a-odd}). The last results of the paper extend the
aforementioned partitions of $\nums\cap \llbracket 0,(a+2)L\rrbracket$
in terms of $\nums^{\ell}$ and $\nums_d$ to the whole set
$\ulf(\nums)$ (see Propositions~\ref{pr:partitionSell}
and~\ref{prop:Sd}).

Our interest in the specific and simple semigroup $\nums=\se
a,a+1,a+2\sd$ was awakened studying the work of Moh in \cite{moh1}
(see also \cite{mss}, for a recent related paper). In his paper, Moh
is centred in finding prime ideals with a minimal generating set of
cardinality arbitrarily high in a power series ring $\mbk\llbracket
x,y,z\rrbracket$ in three variables, generalising a classic work of
Macaulay (see, e.g, the paper of Abhyankar on Macaulay's examples
\cite{abhyankar}). In fact, the work of Moh is a natural extension of
that of Herzog in \cite{herzog}, where he proves that the definition
ideals of the semigroup ring $\mbk[t^a,t^b,t^c]$ can be generated by
two or three binomials. Moh's primes are the definition ideal of a
``quasi-monomial'' curve of the type
$(t^{ab}+t^{ab+\lambda},t^{(a+1)b},t^{(a+2)b})$, where $a$ is an odd
natural number. In his work, Moh studies the numerical semigroup
generated by $a+1,a+2$ and uses an ordering given by the mapping
$\sigma:\mbk\llbracket x,y,z\rrbracket\to \mbk\llbracket
x,y,z\rrbracket$, $\sigma(x)=x^{a}$, $\sigma(y)=y^{a+1}$,
$\sigma(z)=z^{a+2}$, $a$ odd. In his proofs, Moh needs to calculate
the denumerant of $r\in\nums$, that is, the dimension of the vector
space $W_r$, and the dimension of a subspace $V_r\subseteq W_r$
attached to the definition ideal of the given affine curve. In fact,
Moh essentially works with the elements of $\nums$ involved in
Theorem~4.3, the key result in \cite{moh1}. To our purposes, we are
mainly interested in the elements of $\nums\cap \llbracket
0,\la)\!)$. To undertake these calculations and some possible
variants, and to deeply understand the reason why they work, it is
very convenient to have the elements of the semigroup $\nums$
stratified according to their length and denumerant. Results like
Corollaries~\ref{cor-length} and~\ref{cor-dim} provide us with the
desired partitions.

There is a very extensive bibliography on (the number of)
factorizations and their lengths. See, e.g., and just to mention a few
of them, \cite{cglm, cggr,gos, mo} and, particularly, \cite{rg} and
the references therein. Note that the semigroup $\nums$ can be seen
threefold: as a numerical semigroup generated by an interval, as a
numerical semigroup generated by an arithmetic sequence, or even as an
embedding dimension three numerical semigroup (see, e.g. \cite{ag,
  agl, gou, goy, gr}). Some of the results presented here could be
deduced from previous works, for instance, in \cite[Corollary~2]{gr},
Garc\'{\i}a-S\'anchez and Rosales give a solution to the membership
problem for numerical semigroups generated by intervals. On the other
hand, in \cite{ag} and \cite{agl}, Aguil\'o, Garc\'{\i}a-S\'anchez and
Llena are able to give a constructive way of finding the set of
factorizations of an embedding dimension three numerical semigroup
using the so-called ``basic factorization''. The present note differs
from the above works in that our object of study is focused on the
semigroup $\nums$, so we can deepen on theses questions. On the other
hand, the introduction and study of the set $\ulf(\numGen)$ of
elements of a numerical semigroup $\numGen$ with a unique length
factorization seems to be a new and original object of study.

\section{Elements with unique length factorization}\label{sec-first}

Let $\nums=\se a,a+1,a+2\sd$ be the numerical semigroup generated by
$a,a+1,a+2$, where $a\in\mbn$, $a\geq 3$.  We begin by proving that
all the factorizations of $r\in\nums$, $r$ smaller enough, have length
$\lfloor r/a\rfloor$.

\begin{notation}
Let $\la:=\lfloor a/2\rfloor (a+2)$, if $a$ is even, or $\la:=(\lfloor
a/2\rfloor +2)a$, if $a$ is odd.  Clearly, $\la\in\nums$. Note that:
\begin{eqnarray*}
\frob(\nums)=\left\lfloor\frac{a}{2}\right\rfloor a-1<\left\lfloor
\frac{a}{2}\right\rfloor (a+2) \leq\left(\left\lfloor
\frac{a}{2}\right\rfloor +1\right)a< \left(\left\lfloor
\frac{a}{2}\right\rfloor +2\right)a\leq (a+1)+\left\lfloor
\frac{a}{2}\right\rfloor \left(a+2\right).
\end{eqnarray*}
\end{notation}

\begin{theorem}\label{theolength}
There is an inclusion $\nums\cap \llbracket
0,\la)\!)\subseteq\ulf(\nums)$, where
$\la\not\in\ulf(\nums)$. Moreover, $\setl(r,\nums)=\{\ell_r\}$, for
every $r\in\nums\cap \llbracket 0,\la)\!)$.
\end{theorem}

The proof of the proposition is a consequence of the following lemma. 

\begin{lemma}\label{lemma-length}
Let $r\in\nums$, $r\neq 0$, and $\alpha\in\fac(r,\nums)$. Then
\begin{itemize}
\item[$(1)$] $|\alpha|a\leq r\leq |\alpha|(a+2)$;
\item[$(2)$] $|\alpha|\leq \lfloor r/a\rfloor$;
\item[$(3)$] $\lfloor r/a\rfloor \leq
  |\alpha|+\sum_{j=0}^{|\alpha|-1}\lfloor
  (2|\alpha|+aj)/a|\alpha|\rfloor$.
\end{itemize}
Furthermore,
\begin{itemize}
\item[$(4)$] If $r\leq\frob(\nums)$, then $|\alpha|=\lfloor r/a\rfloor
  $ and $\lfloor r/a\rfloor <\lfloor a/2\rfloor$.
\item[$(5)$] If $r>\frob(\nums)$, then $\lfloor a/2\rfloor \leq |\alpha|$.
\item[$(6)$] If $r<\lfloor a/2\rfloor (a+2)$, then $\lfloor r/a\rfloor
  \leq \lfloor a/2\rfloor$.
\item[$(7)$] If $\frob(\nums)<r<\lfloor a/2\rfloor (a+2)$, then
  $|\alpha|=\lfloor r/a\rfloor =\lfloor a/2\rfloor$.
\end{itemize}
Suppose that $a$ is even, so $\la=\lfloor a/2\rfloor (a+2)$.
\begin{itemize}  
\item[$(8)$] If $r\geq \la$, then $\lfloor r/a\rfloor >\lfloor
  a/2\rfloor$.  If $r=\la$, then $(0,0,\lfloor a/2\rfloor )$ and
  $(\lfloor a/2\rfloor +1,0,0)$ are two factorizations of $r$ of
  different length.
\end{itemize}
Suppose that $a$ is odd, so $\la=(\lfloor a/2\rfloor +2)a$.
\begin{itemize}
\item[$(9)$] If $r=\lfloor a/2\rfloor (a+2)$, then $|\alpha|=\lfloor
  r/a\rfloor =\lfloor a/2\rfloor$.
\item[$(10)$] If $\lfloor a/2\rfloor (a+2)<r<\la$, then
  $|\alpha|=\lfloor r/a\rfloor =\lfloor a/2\rfloor +1$.
\item[$(11)$] If $r\geq \la$, then $\lfloor r/a\rfloor >\lfloor
  a/2\rfloor +1$.  If $r=\la$, then $(0,1,\lfloor a/2\rfloor )$ and
  $(\lfloor a/2\rfloor +2,0,0)$ are two factorizations of $r$ of
  different length.
\end{itemize}
\end{lemma}
\begin{proof}[Proof of the Lemma]
Clearly, $|\alpha|a\leq\alpha_1a+\alpha_2(a+1)+\alpha_3(a+2)=r\leq
|\alpha|(a+2)$. Dividing by $a$ the first inequality, $|\alpha|\leq
r/a$ and $|\alpha|\leq \lfloor r/a\rfloor$. This shows items $(1)$
and $(2)$.  Dividing by $a$ the second inequality, $r/a\leq
|\alpha|(a+2)/a$ and $\lfloor r/a\rfloor \leq \lfloor
|\alpha|(a+2)/a\rfloor$.  Let $m\in\mbn$, $m\geq 1$. By Hermite's
identity,
\begin{eqnarray}\label{hermite}
\left\lfloor m\frac{a+2}{a}\right\rfloor =\sum_{j=0}^{m-1}\left\lfloor
\frac{a+2}{a}+\frac{j}{m}\right\rfloor =\sum_{j=0}^{m-1}\left\lfloor
\frac{am+2m+aj}{am}\right\rfloor =m+\sum_{j=0}^{m-1}\left\lfloor
\frac{2m+aj}{am}\right\rfloor .
\end{eqnarray}  
Substituting $m$ by $|\alpha|\neq 0$ in the equality \eqref{hermite},
we deduce the third item.

Suppose that $r\leq\frob(\nums)=\lfloor a/2\rfloor a-1$. So $r<\lfloor
a/2\rfloor a$ and $\lfloor r/a\rfloor <\lfloor a/2\rfloor$. Then, by
item $(2)$,
\begin{eqnarray}\label{rsmall}
|\alpha|\leq \left\lfloor \frac{r}{a}\right\rfloor <\left\lfloor
\frac{a}{2}\right\rfloor .
\end{eqnarray}
For every $0\leq j\leq |\alpha|-1$, we have $2|\alpha|+aj<2\lfloor
a/2\rfloor +a|\alpha|-a\leq a|\alpha|$.  Thus, for $m=|\alpha|$, we
get in the equality \eqref{hermite}: $\lfloor |\alpha|(a+2)/a\rfloor
=|\alpha|$. Hence, $\lfloor r/a\rfloor \leq \lfloor
|\alpha|(a+2)/a\rfloor =|\alpha|$. Using the inequalities
\eqref{rsmall}, we conclude $|\alpha|=\lfloor r/a\rfloor <\lfloor
a/2\rfloor$, which proves the fourth item.

Suppose that $r>\frob(\nums)=\lfloor a/2\rfloor a-1$. By item $(1)$,
$r\leq |\alpha|(a+2)$. Therefore, $\lfloor a/2\rfloor a\leq
|\alpha|(a+2)$. If $|\alpha|\leq \lfloor a/2\rfloor -1$, then $\lfloor
a/2\rfloor a\leq |\alpha|(a+2)\leq (\lfloor a/2\rfloor -1)(a+2)\leq
\lfloor a/2\rfloor a-2$, a contradiction. Thus, $\lfloor a/2\rfloor
\leq |\alpha|$, which proves item $(5)$.  If $r<\lfloor a/2\rfloor
(a+2)$, then $\lfloor r/a\rfloor \leq r/a<\lfloor a/2\rfloor
(a+2)/a\leq \lfloor a/2\rfloor +\lfloor a/2\rfloor (2/a)\leq \lfloor
a/2\rfloor +1$ and $\lfloor r/a\rfloor \leq \lfloor a/2\rfloor$. This
proves item $(6)$. Item $(7)$ follows from items $(2)$, $(5)$ and
$(6)$.

Suppose that $a$ is even and $r\geq \lfloor a/2\rfloor (a+2)$. Then
$\lfloor a/2\rfloor =a/2$ and $r\geq \lfloor a/2\rfloor a+2\lfloor
a/2\rfloor =\lfloor a/2\rfloor a+a$. So $r/a\geq \lfloor a/2\rfloor
+1$ and $\lfloor r/a\rfloor \geq \lfloor a/2\rfloor +1>\lfloor
a/2\rfloor$.  If $r=\lfloor a/2\rfloor (a+2)$, then clearly
$(0,0,\lfloor a/2\rfloor )$ and $(\lfloor a/2\rfloor +1,0,0)$ are two
factorizations of $r$ with different length. This proves item $(8)$.

Suppose that $a$ is odd and that $r=\lfloor a/2\rfloor (a+2)$. Thus,
$r/a=\lfloor a/2\rfloor (a+2)/a$.  Set $m=\lfloor a/2\rfloor
=(a-1)/2\geq 1$. By equation \eqref{hermite},
\begin{eqnarray*}
\left\lfloor \frac{r}{a}\right\rfloor =\left\lfloor
m\frac{a+2}{a}\right\rfloor =m+\sum_{j=0}^{m-1} \left\lfloor
\frac{2m+aj}{am}\right\rfloor .
\end{eqnarray*}
For $0\leq j\leq \lfloor a/2\rfloor -1$, then $2m+aj\leq 2\lfloor
a/2\rfloor +a\lfloor a/2\rfloor -a=(a-1)+a\lfloor a/2\rfloor
-a=a\lfloor a/2\rfloor -1=am-1$. Thus, $\lfloor (2m+aj)/am\rfloor =0$
and $\lfloor r/a\rfloor =m=\lfloor a/2\rfloor$. By $(2)$,
$|\alpha|\leq \lfloor r/a\rfloor$ and, by $(5)$, $\lfloor a/2\rfloor
\leq |\alpha|$. Hence $|\alpha|=\lfloor r/a\rfloor =\lfloor a/2\rfloor
$. This proves item $(9)$.

Suppose that $a$ is odd and that $\lfloor a/2\rfloor (a+2)<r<a(\lfloor
a/2\rfloor +2)$. By $(1)$, $\lfloor a/2\rfloor (a+2)<r\leq
|\alpha|(a+2)$ and $\lfloor a/2\rfloor <|\alpha|$.  Since $r<a(\lfloor
a/2\rfloor +2)$, then $\lfloor r/a\rfloor \leq r/a<\lfloor a/2\rfloor
+2$, so $\lfloor r/a\rfloor \leq \lfloor a/2\rfloor +1$. By $(2)$,
$|\alpha|\leq \lfloor r/a\rfloor$. Therefore, $\lfloor a/2\rfloor
+1\leq |\alpha|\leq \lfloor r/a\rfloor \leq \lfloor a/2\rfloor +1$ and
$|\alpha|=\lfloor r/a\rfloor =\lfloor a/2\rfloor +1$, which shows item
$(10)$.

Suppose that $a$ is odd and $r\geq a(\lfloor a/2\rfloor +2)$. Then
$r/a\geq \lfloor a/2\rfloor +2$ and $\lfloor r/a\rfloor >\lfloor
a/2\rfloor +1$.  If $r=a(\lfloor a/2\rfloor +2)$, then clearly
$(0,1,\lfloor a/2\rfloor )$ and $(\lfloor a/2\rfloor +2,0,0)$ are two
factorizations of $r$ of different length. This shows item $(11)$ and
finishes the whole proof.
\end{proof}

\begin{proof}[Proof of Theorem~\ref{theolength}]
If $r=0$, then $\alpha=0$ and $|\alpha|=\lfloor r/a\rfloor$. If
$r\in\nums$, with $0<r<\la$, and $\alpha\in\fac(r,\nums)$, then
$|\alpha|=\lfloor r/a\rfloor$ by items $(4)$, $(7)$, $(9)$ and $(10)$
of Lemma~\ref{lemma-length}. If $r=\la$, then $r$ has two
factorizations of different length by items $(8)$ and $(11)$ of the
same lemma.
\end{proof}

\begin{example}\label{examplen3}
Let $a=3$ and $\nums=\se 3,4,5\sd$. Then $\frob(\nums)=\lfloor
a/2\rfloor a-1=2$, $\lfloor a/2\rfloor (a+2)=5$ and $\la=(\lfloor
a/2\rfloor +2)a=9$. Here $\fac(0,\nums)=\{(0,0,0)\}$,
$\fac(3,\nums)=\{(1,0,0)\}$, $\fac(4,\nums)=\{(0,1,0)\}$,
$\fac(5,\nums)=\{(0,0,1)\}$, $\fac(6,\nums)=\{(2,0,0)\}$,
$\fac(7,\nums)=\{(1,1,0)\}$, $\fac(8,\nums)=\{(1,0,1),(0,2,0)\}$ and
$\fac(9,\nums)=\{(3,0,0),(0,1,1)\}$.  For every $r\in\nums$, $r\leq
8$, and for every $\alpha\in\fac(r,\nums)$, then $|\alpha|=\lfloor
r/a\rfloor$.  This is no longer true for $9$, since $|(3,0,0)|=3$ and
$|(0,1,1)|=2$, where $(3,0,0),(0,1,1)\in\fac(9,\nums)$.

We can swiftly do these computations with the \texttt{numericalsgps}
\cite{numericalsgps} \texttt{GAP} \cite{gap} package.
\begin{verbatim}
gap> s:=NumericalSemigroup(3,4,5);;
gap> List(Intersection([0..9],s),x->[x,Factorizations(x,s)]);
[ [ 0, [ [ 0, 0, 0 ] ] ], [ 3, [ [ 1, 0, 0 ] ] ], [ 4, [ [ 0, 1, 0 ] ] ],
  [ 5, [ [ 0, 0, 1 ] ] ], [ 6, [ [ 2, 0, 0 ] ] ], [ 7, [ [ 1, 1, 0 ] ] ],
  [ 8, [ [ 0, 2, 0 ], [ 1, 0, 1 ] ] ], [ 9, [ [ 3, 0, 0 ], [ 0, 1, 1 ] ] ] ]   
\end{verbatim}
\end{example}

\begin{question}
Given a numerical semigroup $\numGen$, what is the maximum element
$\msl\in\numGen$ such that $\numGen\cap \llbracket 0,\msl)\!)\subseteq
\ulf(\numGen)$ and $\msl\not\in \ulf(\numGen)$?  We come back to this
question in Section~\ref{sec-Betti} (see
Corollary~\ref{cor-minbetti}).
\end{question}

We need two preliminary results before introducing some more notations.

\begin{proposition}\label{prop-omega}
Let $r\in\nums\cap \llbracket 0,\la)\!)$. Let
$\alpha,\gamma\in\fac(r,\nums)$, $\alpha\neq\gamma$.  Set
$\omega:=(-1,2,-1)\in\mbz^3$. Then:
\begin{itemize}
\item[$(1)$] $\alpha_1+\alpha_2+\alpha_3=\gamma_1+\gamma_2+\gamma_3$;
\item[$(2)$] $\alpha_2+2\alpha_3=\gamma_2+2\gamma_3$;
\item[$(3)$] $\alpha_1\neq\gamma_1$, $\alpha_2\neq\gamma_2$ and
  $\alpha_3\neq\gamma_3$.
\item[$(4)$] If $\gamma_2>\alpha_2$, then $\gamma=\alpha+j\omega$, for
  some $1\leq j\leq\min(\alpha_1,\alpha_3)$.\\ If $\gamma_2<\alpha_2$,
  then $\alpha=\gamma+j\omega$, for some $1\leq
  j\leq\min(\gamma_1,\gamma_3)$.
\end{itemize}
\end{proposition}
\begin{proof}
By Theorem~\ref{theolength}, $|\alpha|=|\gamma|=\lfloor
r/a\rfloor$. So
$\alpha_1+\alpha_2+\alpha_3=|\alpha|=|\gamma|=\gamma_1+\gamma_2+\gamma_3$.
In particular,
$\alpha_2+2\alpha_3=r-|\alpha|a=r-|\gamma|a=\gamma_2+2\gamma_3$. This
shows items $(1)$ and $(2)$. If $\alpha_j=\gamma_j$, for some
$j\in\{1,2,3\}$, using items $(1)$ and $(2)$, it would follow
$\alpha=\gamma$. This proves item $(3)$. Suppose that
$\gamma_2>\alpha_2$. By $(2)$,
$2\gamma_3=\alpha_2-\gamma_2+2\alpha_3<2\alpha_3$, so
$\gamma_3<\alpha_3$. Using items $(1)$ and $(2)$,
$\alpha_1-\gamma_1=(\gamma_2-\alpha_2)+(\gamma_3-\alpha_3)=
(2\alpha_3-2\gamma_3)+(\gamma_3-\alpha_3)=\alpha_3-\gamma_3>0$ and
$\alpha_1-\gamma_1>0$. Let $j:=\alpha_1-\gamma_1=\alpha_3-\gamma_3\geq
1$. In particular, $\gamma_1=\alpha_1-j$ and
$\gamma_3=\alpha_3-j$. Since $\gamma_1,\gamma_3\geq 0$, we obtain
$j\leq\min(\alpha_1,\alpha_3)$. Moreover, from $(1)$,
$\gamma_2=\alpha_2+(\alpha_1-\gamma_1)+(\alpha_3-\gamma_3)=
\alpha_2+2j$. Hence, $\gamma=\alpha+j(-1,2,-1)=\alpha+j\omega$, where
$1\leq j\leq\min(\alpha_1,\alpha_3)$.  The proof of the case
$\gamma_2<\alpha_2$ is analogous.
\end{proof}

\begin{lemma}\label{lemma-element}
Let $r\in\mbn$. Let $u,v\in\mbn$ be such that $r=au+v$. Let
\begin{eqnarray*}
  \varphi=(\varphi_1,\varphi_2,\varphi_3):= \left(u-\left\lfloor
  \frac{v+1}{2}\right\rfloor ,v-2\left\lfloor \frac{v}{2}\right\rfloor
  , \left\lfloor \frac{v}{2}\right\rfloor \right)\in\mbz^3.
\end{eqnarray*}
Then, $\varphi_1a+\varphi_2(a+1)+\varphi_3(a+2)=r$. Moreover,
$\varphi\in\mbn^3$ if and only if $v\leq 2u$. In particular, if
$\varphi\in\mbn^3$, or equivalently $v\leq 2u$, then $r\in\nums$ and
$\varphi\in\fac(r,\nums)$.
\end{lemma}
\begin{proof}
A simple computation gives
$\varphi_1a+\varphi_2(a+1)+\varphi_3(a+2)=r-a\lfloor (v+1)/2\rfloor
+av-a\lfloor v/2\rfloor$.  Since $\lfloor v/2\rfloor +\lfloor
(v+1)/2\rfloor =v$, we obtain $-a\lfloor (v+1)/2\rfloor +av-a\lfloor
v/2\rfloor =0$ and $\varphi_1a+\varphi_2(a+1)+\varphi_3(a+2)=r$. Note
that $2\lfloor v/2\rfloor \leq v$ and $v-2\lfloor v/2\rfloor \geq
0$. Suppose that $\varphi\in\mbn^3$. So $\lfloor (v+1)/2\rfloor \leq
u$. If $v$ is odd, then $v<v+1\leq 2u$; if $v$ is even, then $v\leq
2u$. In any case, $v\leq 2u$. Conversely, if $v\leq 2u$, then
$(v+1)/2\leq (2u+1)/2$ and $\lfloor (v+1)/2\rfloor \leq \lfloor
(2u+1)/2\rfloor =u$.  Thus, $\varphi\in\mbn^3$. Note that, in this
case, $r\in\nums$ and $\varphi\in\fac(r,\nums)$.
\end{proof}

\begin{remark}
Clearly $r$ might be in $\nums$, whereas $\varphi\not\in\mbn^3$. For
instance, take $a=3$, $\nums=\se 3,4,5\sd$ and $r=9$. Then, $r=au+v$,
with $u=1$ and $v=6$ and $\varphi=(u-\lfloor (v+1)/2\rfloor
,v-2\lfloor v/2\rfloor ,\lfloor v/2\rfloor )=(-2,0,3)\not\in\mbn^3$.
\end{remark}

Let us recall and introduce some new notations. 

\begin{notation}\label{notation-ell}
Let $r\in\mbn$. Let 
\begin{itemize}
\item $\ell_r:=\lfloor r/a\rfloor$, which is the quotient of the
  Euclidian division in $\mbn$ of $r$ by $a$;
\item $\rem_r:=r-a\ell_r$ be the remainder of the Euclidian division
  in $\mbn$ of $r$ by $a$; we will also denote this remainder by $r\bmod a$.
\end{itemize}
Thus, $r=a\ell_r+\rem_r$. By Theorem~\ref{theolength}, if
$r\in\nums\cap\llbracket 0,\la)\!)$, then $\setl(r,\nums)=\{\ell_r\}$.
In other words, $\ell_r$ is the length of any factorization of $r$,
when $r\in\nums\cap\llbracket 0,\la)\!)$. Let
\begin{itemize}
\item $\phi_r:=(\phi_{r,1},\phi_{r,2},\phi_{r,3}):=
  \left(\ell_r-\left\lfloor \frac{\rem_r+1}{2}\right\rfloor ,
  \rem_r-2\left\lfloor \frac{\rem_r}{2}\right\rfloor ,\left\lfloor
  \frac{\rem_r}{2} \right\rfloor \right)\in\mbz^3.$
\end{itemize}
We denote $\phi_r$ the \emph{seed} vector of $r$. Note that
$\phi_{r,2}+2\phi_{r,3}=\rem_r$ and
$\phi_{r,1}+\phi_{r,2}+\phi_{r,3}=\ell_r$.

By the previous lemma, $\phi_r\in\mbn^3$ if and only if $\rem_r\leq
2\ell_r$ and, in such a case, $r\in\nums$ and
$\phi_r\in\fac(r,\nums)$.

Suppose now that $\phi_r\in\mbn^3$, so $r\in\nums$ and
$\phi_r\in\fac(r,\nums)$. In this setting, $\phi_r$ is also called the
\emph{seed factorization} of $r$. We also denote:
\begin{itemize}
\item $\kappa_r:=\min(\phi_{r,1},\phi_{r,3})$,
  $\xi_r:=\max(\phi_{r,1},\phi_{r,3})$,
  $\iota_r:=\phi_{r,2}+|\phi_{r,1}-\phi_{r,3}|$ and
  $c_{r}:=\phi_{r,3}-\phi_{r,1}$.
\end{itemize}
Note that $\kappa_r$, $\xi_r$ and $\iota_r\in\mbn$ and $c_r\in\mbz$.
Moreover, $c_r=\rem_r-\ell_r$.
\begin{itemize}
\item $\mcb_r:=\{\mon^{\phi_r+j\omega} :  0\leq j\leq\kappa_r\}=
  \{x^{\phi_{r,1}-j}y^{\phi_{r,2}+2j}z^{\phi_{r,3}-j}  :  0\leq
  j\leq\kappa_r\}$ (ordered subset).
\end{itemize}
Recall that, for any $r\in\nums$,
\begin{itemize}
\item $W_r:=\se \mon^{\alpha} :  \alpha\in\fac(r,\nums)\sd$ and
  $\delta_r:=\dim W_r$, where $\delta_r\geq 1$.
\end{itemize}
Given $i\in\mbn$, let $\Gamma_i$ be defined as
\begin{itemize}
\item $\Gamma_0:=\{0\}$, for $i=0$; $\Gamma_1:=\{-1,0,1\}$, for $i=1$,
  and $\Gamma_{i}:=\{-i,-i+1,i-1,i\}$, for $i\geq 2$.
\end{itemize}
We claim that $c_r\in\Gamma_{\iota_r}$. Indeed, since
$\phi_{r,2}=\rem_r-2\lfloor \rem_r/2\rfloor$, if $\rem_r$ is even,
$\phi_{r,2}=0$, and if $\rem_r$ is odd, $\phi_{r,2}=1$. Since
$\iota_r=\phi_{r,2}+|\phi_{r,1}-\phi_{r,3}|$, if $\iota_r=0$, we
deduce $\phi_{r,2}=0$, $\phi_{r,1}=\phi_{r,3}$ and
$c_{r}=\phi_{r,3}-\phi_{r,1}=0$, so $c_{r}\in\Gamma_0$. Suppose that
$\iota_r=1$. Then, either $\phi_{r,2}=1$, $\phi_{r,1}=\phi_{r,3}$, and
$c_{r}=\phi_{r,3}-\phi_{r,1}=0$, or else, $\phi_{r,2}=0$,
$|\phi_{r,1}-\phi_{r,3}|=1$, and
$c_{r}=\phi_{r,3}-\phi_{r,1}\in\{-1,1\}$. In any case,
$c_{r}\in\Gamma_1$. Finally, suppose that $\iota_r\geq 2$.  If
$\phi_{r,2}=1$, then $|\phi_{r,1}-\phi_{r,3}|=\iota_r-1$ and
$c_{r}=\phi_{r,3}-\phi_{r,1}\in\{-\iota_r+1,\iota_r-1\}$. If
$\phi_{r,2}=0$, then $|\phi_{r,1}-\phi_{r,3}|=\iota_r$ and
$c_{r}=\phi_{r,3}-\phi_{r,1}\in\{-\iota_r,\iota_r\}$. We conclude that
$c_{r}\in\Gamma_{\iota_r}$.
\end{notation}

\section{The membership problem and the set of
  factorizations}\label{sec-length}

Let $\nums=\se a,a+1,a+2\sd$ be the numerical semigroup generated by
$a,a+1,a+2$, where $a\in\mbn$, $a\geq 3$.  In this section we first
study the membership problem in terms of $\ell_r$ and the seed vector
$\phi_r$. See \cite[Corollary~2]{rg}, for a resolution of the
membership problem for numerical semigroups generated by intervals in
terms of $\lfloor r/a\rfloor$, and \cite[Theorem~2]{ag}, for a
description of the full set of factorizations of an embedding
dimension three numerical semigroup.

\begin{theorem}\label{theo-length}
Let $r\in\mbn$. The following conditions are equivalent:
\begin{eqnarray*}
(i)\phantom{+} r\in\nums;\phantom{+} (ii)\phantom{+} \rem_r\leq
  2\ell_r;\phantom{+} (iii)\phantom{+} \phi_r\in\mbn^3;\phantom{+}
  (iv)\phantom{+}\mbox{there exist }u,v\in\mbn,\; v\leq 2u,\; r=au+v.
\end{eqnarray*}
In such a case, $\phi_r\in\fac(r,\nums)$. Suppose that 
$r\in\nums\cap\llbracket 0,\la)\!)$, so $\setl(r,\nums)=\{\ell_r\}$.
\begin{itemize}
\item[$(1)$] If $u,v\in\mbn$ are such that $v\leq 2u$ and
$r=au+v$, then $(u,v)=(\ell_r,\rem_r)$;
\item[$(2)$] $\fac(r,\nums)=\{\phi_r+j\omega :  0\leq
  j\leq\kappa_r\}$;
\item[$(3)$] $\mcb_r=\{\mon^{\phi_r+j\omega} :  0\leq
  j\leq\kappa_r\}$ is a basis of $W_r$. In particular,
  $\delta_r=\kappa_r+1$.
\item[$(4)$] $\ell_r=2(\delta_r-1)+\iota_r$. In particular, $\ell_r$
  and $\iota_r$ have the same parity and $0\leq \iota_r\leq\ell_r$.
\end{itemize}
\end{theorem}
\begin{proof}
Clearly $(ii)$ implies $(iv)$. By Lemma~\ref{lemma-length}, $(ii)$ is
equivalent to $(iii)$, which, in turn, implies $(i)$ and
$\phi_{r}\in\fac(r,\nums)$. Moreover, if $(iv)$ holds, by
Lemma~\ref{lemma-length}, $\varphi=(u-\lfloor (v+1)/2\rfloor
,v-2\lfloor v/2\rfloor ,\lfloor v/2\rfloor )\in\mbn^3$ and
$\varphi_1a+\varphi_2(a+1)+\varphi_3(a+2)=r$, so $r\in\nums$ and
consequently $(iv)$ implies $i$.  Thus, in order to prove the first
part, it remains to show $(i)$ implies $(ii)$. Suppose that $(i)$
holds and that $r<\la$.  Since $r\in\nums$, there exists
$\alpha\in\fac(r,\nums)$. By Theorem~\ref{theolength},
$|\alpha|=\lfloor r/a\rfloor =\ell_r$ and
$r=\alpha_1a+\alpha_2(a+1)+\alpha_3(a+2)=|\alpha|a+\alpha_2+2\alpha_3=
a\ell_r+(\alpha_2+2\alpha_3)$.  Hence, $\rem_r=\alpha_2+2\alpha_3\leq
2(\alpha_1+\alpha_2+\alpha_3)=2|\alpha|=2\ell_r$ and $\rem_r\leq
2\ell_r$. Suppose that $r\geq\la$. In particular, since
$\la>\frob(\nums)$, we deduce $r\in\nums$. Let us show that if
$r\geq\la$, then $(ii)$ always holds. Indeed, by
Lemma~\ref{lemma-length}, $(8)$ and $(11)$, $\lfloor a/2\rfloor
<\lfloor r/a\rfloor$. Thus, if $a$ is even, then $\rem_r<a=2\lfloor
a/2\rfloor <2\lfloor r/a\rfloor =2\ell_r$ and, if $a$ is odd, then
$\rem_r\leq a-1=2\lfloor a/2\rfloor <2\lfloor r/a\rfloor =2\ell_r$, as
well.

Suppose now that $r\in\nums\cap\llbracket 0,\la)\!)$.  In particular
$\phi_r\in\mbn^3$ and $\phi_r\in\fac(r,\nums)$.  Let $u,v\in\mbn$
satisfy $v\leq 2u$ and $r=au+v$. By Lemma~\ref{lemma-element},
$\varphi:=(u-\lfloor (v+1)/2\rfloor ,v-2\lfloor v/2\rfloor ,\lfloor
v/2\rfloor )\in\mbn^3$ and $\varphi\in\fac(r,\nums)$.  Note that
$\phi_{r,2}=\rem_r-2\lfloor \rem_r/2\rfloor$ and $\varphi_2=v-2\lfloor
v/2\rfloor$ are either 1 or 0. If they were distinct, then, by
Proposition~\ref{prop-omega}, $|\phi_{r,2}-\varphi_2|=2j$, for some
$j\geq 1$, a contradiction. Thus, $\phi_{r,2}=\varphi_2$ and, by
Proposition~\ref{prop-omega}, $\phi_r=\varphi$.  Therefore, $v=\rem_r$
and $u=\ell_r$. This proves item $(1)$.

Since $\phi_r\in\mbn^3$, by the definition of $\kappa_r$,
$\phi_r+j\omega\in\mbn^3$, for every $1\leq j\leq\kappa_r$. As
$\phi_r\in\fac(r,\nums)$, it easily follows that $\phi_r+j\omega$ is a
factorization of $r$, for $-ja+j2(a+1)-j(a+2)=0$.  Conversely, let
$\gamma\in\fac(r,\nums)$, $\gamma\neq\phi_r$.  By
Proposition~\ref{prop-omega}, $\gamma_2\neq \phi_{r,2}$ and
$\gamma_2=\phi_{r,2}\pm 2j$, for some $j\geq 1$. Since $\phi_{r,2}=0$
or 1, and $\varphi_2\in\mbn$, we have that, necessarily,
$\gamma_2=\phi_{r,2}+2j$ and $\gamma=\phi_r+j\omega$, for some $1\leq
j\leq\min(\phi_{r,1},\phi_{r,3})=\kappa_r$. This shows item $(2)$.
Clearly, item $(3)$ follows from item $(2)$.

Finally, let us prove item $(4)$. Indeed,
$\ell_r=\phi_{r,1}+\phi_{r,2}+\phi_{r,3}$, where
\[
\phi_{r,1}+\phi_{r,3}=\max(\phi_{r,1},\phi_{r,3})+\min(\phi_{r,1},\phi_{r,3})=
\xi_r+\kappa_r=2\kappa_r+(\xi_r-\kappa_r)=2\kappa_r+|\phi_{r,1}-\phi_{r,3}|.
\]
By item $(3)$, $\kappa_r=\delta_r-1$. Thus,
$\ell_r=2\kappa_r+|\phi_{r,1}-\phi_{r,3}|+\phi_{r,2}=2(\delta_r-1)+\iota_r$.
\end{proof}

\begin{remark}\label{boundhypothesis}
The hypothesis $r<\la$ is essential in Theorem~\ref{theo-length},
$(1)$.  For instance, take $a=10$, so $\la=\lfloor a/2\rfloor
(a+2)=60$, and $r=60$. Then $r=10\cdot 5+10$, $u:=5$, $v:=10$, where
$v\leq 2u$, but $u\neq\ell_r=\lfloor r/a\rfloor =6$ and $v\neq
\rem_r=0$. However, the containment $\fac(r,\nums)\supseteq
\{\phi_r+j\omega : 0\leq j\leq\kappa_r\}$ still holds for any
$r\in\nums$, even without the restriction $r<\la$. Indeed, if
$r\in\nums$, then by the equivalence $(i)$, $(iii)$ of
Theorem~\ref{theo-length}, $\phi_r\in\mbn^3$ and
$\phi_r+j\omega\in\fac(r,\nums)$, $0\leq j\leq\kappa_r$, as shown in
the proof of Theorem~\ref{theo-length}, where there is no necessity of
a bound on $r$.  In particular, $\{\mon^{\phi_r+j\omega} : 0\leq j\leq
\kappa_r\}$ is a set of linearly independent elements of $W_r$.
\end{remark}

Let $L:=\lfloor (a-1)/2\rfloor$. As a corollary of
Theorem~\ref{theo-length} we give a natural partition of $\nums\cap
\llbracket 0,(a+2)L\rrbracket$ in terms of the subsets
$\nums^{\ell}=\{r\in\nums : \setl(r,\nums)=\{\ell\}\}$.

\begin{corollary}\label{cor-length}
Let $\ell\in\mbn$, with $0\leq \ell\leq L$. The following
hold.
\begin{itemize}
\item[$(1)$] $(a+2)L<\la$ and $\nums^{\ell}\subseteq \nums\cap
  \llbracket 0,\la)\!)$.
\item[$(2)$] $\nums^{\ell}=\llbracket a\ell,(a+2)\ell\rrbracket$.
\item[$(3)$] $\nums\cap \llbracket
  0,(a+2)L\rrbracket=\bigcup_{\ell=0}^{L}\nums^{\ell}$, where
  $\nums^{0},\nums^1,\ldots,\nums^L$ are pairwise disjoint.
\end{itemize}
\end{corollary}
\begin{proof}
If $a$ is even, then $(a+2)L=(a+2)(a-2)/2<(a+2)a/2=\la$, and, if $a$
is odd, then $(a+2)L=(a+2)(a-1)/2<a\lfloor (a-1)/2\rfloor+a<(\lfloor
a/2\rfloor +2)a=\la$.  Let $r\in\nums^{\ell}$. Then, $r\in\nums$ and
every $\alpha\in\fac(r,\nums)$ satisfies $|\alpha|=\ell$. By
Lemma~\ref{lemma-length}, $r\leq |\alpha|(a+2)=(a+2)\ell\leq
(a+2)L<\la$.  This proves item $(1)$.

Let $r\in\nums^{\ell}$, in particular, by item $(1)$, $r\in \nums\cap
\llbracket 0,\la)\!)$.  By Theorem~\ref{theolength},
$\ell=|\alpha|=\lfloor r/a\rfloor =\ell_r$, for all
$\alpha\in\fac(r,\nums)$ and, by Lemma~\ref{lemma-length},
$a\ell_r=|\alpha|a\leq r\leq |\alpha|(a+2)=(a+2)\ell_r$. Thus, $r\in
\llbracket a\ell,(a+2)\ell\rrbracket$. Conversely, let $r\in
\llbracket a\ell,(a+2)\ell\rrbracket$.  Then $r=a\ell+v$, for some
$v\in\mbn$, $0\leq v\leq 2\ell$. By the equivalence between $(i)$ and
$(iv)$ in Theorem~\ref{theo-length}, we deduce $r\in\nums$. Note that
$r\leq (a+2)\ell\leq (a+2)L<\la$.  Thus, $r\in\nums\cap\llbracket
0,\la)\!)$, so $\setl(r,\nums)=\{\ell_r\}$.  Moreover, $r=a\ell+v$,
with $0\leq v\leq 2\ell$. By Theorem~\ref{theo-length},
$(\ell,v)=(\ell_r,\rem_r)$. Thus $\ell=\ell_r$, and since
$\setl(r,\nums)=\{\ell_r\}$, it follows that $r\in\nums^{\ell}$. Hence
$\nums^{\ell}=\llbracket a\ell,(a+2)\ell\rrbracket$.  This proves item
$(2)$.

Let $r\in\nums\cap \llbracket 0,(a+2)L\rrbracket$, in particular, by
item $(1)$, $r\in \nums\cap \llbracket 0,\la)\!)$.  By
Theorem~\ref{theolength}, $r\in\ulf(\nums)$ and
$\setl(r,\nums)=\{\ell_r\}$, so $r\in\nums^{\ell_r}$. By item $(2)$,
$\nums^L=[aL,(a+2)L]$ and $s:=(a+2)L\in\nums^L$. By
Theorem~\ref{theolength} again, $s\in\nums\cap\llbracket 0,\la)\!)$,
$\setl(s,\nums)=\{\ell_s\}$ and $\ell_s=L$.  Dividing $0\leq r\leq s$
by $a$, then $\ell_r=\lfloor r/a\rfloor \leq \lfloor s/a\rfloor
=\ell_s=L$ and
$r\in\nums^{\ell_r}\subset\bigcup_{\ell=0}^{L}\nums^{\ell}$. Conversely,
if $r\in \bigcup_{\ell=0}^{L}\nums^{\ell}$, then $r\in\nums^{\ell}$,
for some $\ell\leq L$. By Lemma~\ref{lemma-length}, $(1)$, and for any
$\alpha\in\fac(r,\nums)$, $r\leq |\alpha|(a+2)=\ell(a+2)\leq
(a+2)L$. Hence $r\in\nums\cap\llbracket 0,(a+2)L\rrbracket$. Note
that, by definition, the $\nums^{0},\nums^1,\ldots,\nums^L$ are
pairwise disjoint.
\end{proof}

Next result gives another characterization of $r$ being in $\nums$,
now in terms of the triple $(\delta_r,\iota_r,c_r)$.

\begin{theorem}\label{theo-dim}
Let $r\in \llbracket 0, \la )\!)$. Then: $r\in\nums$ if and only if
\begin{center}
there exist $d,i,c\in\mbz$, $d\geq 1$,
$i\geq 0$ and $c\in\Gamma_i$, such that $r=(a+1)(2d-2+i)+c$.
\end{center}
In such a case, $(d,i,c)=(\delta_r,\iota_r,c_r)$.  In
particular, $2d-2+i=\ell_r$.
\end{theorem}
\begin{proof}
Suppose that $r\in\nums\cap\llbracket 0, \la )\!)$. By
Theorem~\ref{theo-length} and Theorem~\ref{theolength},
$\phi_r\in\fac(r,\nums)$ and $|\phi_r|=\ell_r$. So
\begin{eqnarray*}
r=\phi_{r,1}a+\phi_{r,2}(a+1)+
\phi_{r,3}(a+2)=(a+1)\ell_r+(\phi_{r,3}-\phi_{r,1}),
\end{eqnarray*}
where, by Theorem~\ref{theo-length}, $\ell_r=2\delta_r-2+\iota_r$ and,
by Notation~\ref{notation-ell},
$\phi_{r,3}-\phi_{r,1}=c_{r}\in\Gamma_{\iota_r}$. Therefore,
$r=(a+1)(2\delta_r-2+\iota_r)+c_{r}$, with $\delta_r\geq 1$,
$\iota_r\geq 0$ and $c_{r}\in\Gamma_{\iota_r}$.

Conversely, supose that there exist $d,i\in\mbn$, $d\geq 1$, and
$c\in\Gamma_i$, such that $r=(a+1)(2d-2+i)+c$. Set $u:=2d-2+i$ and
$v:=c+u$.  Then, $r=au+v$, where $u=2(d-1)+i\geq 0$, $v=c+u\geq
-i+2(d-1)+i\geq 0$ and $v=c+u\leq i+u\leq 2u$. By
Theorem~\ref{theo-length}, $r\in\nums$. Let us see that
$(d,i,c)=(\delta_r,\iota_r,c_r)$.  Since $r\in\nums$ and $r<\la$, by
Theorem~\ref{theo-length}, $u=\ell_r$ and $v=\rem_r$.  Thus,
$c=v-u=\rem_r-\ell_r=c_r$ and $c=c_r$. In particular,
$\iota_r=\phi_{r,2}+|\phi_{r,1}-\phi_{r,3}|=\phi_{r,2}+|c_r|=\phi_{r,2}+|c|$,
where $\phi_{r,2}=\rem_r-2\lfloor \rem_r/2\rfloor =v-2\lfloor
v/2\rfloor =c+u-2\lfloor (c+u)/2\rfloor$ is either $0$ or $1$. If
$\phi_{r,2}=0$, i.e., $c+u=c+2d-2+i$ is even, then $c+i$ is
even. Since $c\in\Gamma_i$ and by the definition of the $\Gamma_i$,
then, necessarily, $c\not\in\{-i+1,i-1\}$ and $c=\pm i$. So
$i=|c|=\phi_{r,2}+|c_r|=\iota_r$. If $\phi_{r,2}=1$, i.e.,
$c+u=c+2d-2+i$ is odd, then $c+i$ is odd. Since $c\in\Gamma_i$,
necessarily, $c=-i+1,i-1$, when $i\geq 2$, or $c=0$, when
$i=1$. Therefore, $i=1+|c|=\phi_{r,2}+|c_r|=\iota_r$. Since
$2d-2+i=u=\ell_r=2\delta_r-2+\iota_r$ and $i=\iota_r$, then
$d=\delta_r$. Therefore, $(d,i,c)=(\delta_r,\iota_r,c_r)$.  In
particular, $2d-2+i=2\delta_r-2+\iota_r$ which, by
Theorem~\ref{theo-length}, is equal to $\ell_r$.
\end{proof}

Let $D:=\lfloor (a+3)/4\rfloor$. Recall that $L=\lfloor
(a-1)/2\rfloor$.  As a corollary of Theorem~\ref{theo-dim}, we obtain
a natural partition of $\nums\cap \llbracket 0,(a+2)L\rrbracket$ in
terms of $\nums_d=\{r\in\nums : \card(\fac(r,\nums))=d\}$ and
$\nums_{d,i}:=\{(a+1)(2d-2+i)+c : c\in\Gamma_i\}$.

\begin{corollary}\label{cor-dim}
Let $d\in\mbn$, with $1\leq d\leq
D$. Set $I_d:=L+2-2d$. 
\begin{itemize}
\item[$(1)$] Then, $2D\leq L+2$ and $0\leq I_d\leq L$. 
\item[$(2)$] For $i\in \llbracket 0, I_d\rrbracket$, then
  $2d-2+i\in\llbracket 0,L\rrbracket$, $\nums_{d,i}\subseteq
  \llbracket 0,\la)\!)$ and
  $\nums_{d,i}=\nums_d\bigcap\nums^{2d-2+i}$.
\item[$(3)$] $\nums_d\bigcap
  (\bigcup_{\ell=0}^L\nums^{\ell})=\bigcup_{i=0}^{I_d}\nums_{d,i}$.
\item[$(4)$] $\nums\cap \llbracket 0,(a+2)L\rrbracket=
(\bigcup_{d=1}^{D}\nums_d)\bigcap (\bigcup_{\ell=0}^L\nums^{\ell})$.
\end{itemize}
Let $r_d:=(a+1)(2d-2)$. 
\begin{itemize}
\item[$(5)$] Then, $\nums_{d,0}=\{r_d\}$,
  $r_d=\min\nums_d\cap\left(\bigcup_{\ell=0}^L\nums^{\ell}\right)$ and
  $W_{r_d}=\se x^{d-1}z^{d-1},\ldots,y^{2d-2}\sd$.
\item[$(6)$] Suppose that $r\in\nums_{d,1}$. If $c_r=-1$, then $W_r=
  xW_{r_d}$; if $c_r=0$, then $W_r=yW_{r_d}$; if $c_r=1$, then
  $W_r=zW_{r_d}$.
\item[$(7)$] Suppose that $r\in\nums_{d,i}$, for some $2\leq i\leq
  I_d$.  If $c_r=-i$, then $W_r=x^iW_{r_d}$; if $c_r=-i+1$, then
  $W_r=x^{i-1}yW_{r_d}$; if $c_r=i-1$, then $W_r=yz^{i-1}W_{r_d}$; if
  $c_r=i$, then $W_r=z^iW_{r_d}$.
\end{itemize}
\end{corollary}
\begin{proof}
Note that $D=\lfloor (a+3)/4\rfloor \leq (a+3)/4$, so $2D\leq (a+3)/2$
and $2D\leq \lfloor (a+3)/2\rfloor =\lfloor (a-1)/2\rfloor +2=L+2$. In
particular, $I_d=L+2-2d\geq L+2-2D\geq 0$. On the other hand, since
$d\geq 1$, we obtain $I_d=L+2-2d\leq L$, which proves item $(1)$.

Let $i\in\llbracket 0,I_d\rrbracket$. Since $d\geq 1$ and $0\leq i\leq
I_d$, we have $0\leq 2d-2+i\leq 2d-2+I_d=L$.

Let $r\in\nums_{d,i}$, where $1\leq d\leq D$ and $i\in\llbracket
0,I_d\rrbracket$. Thus,
\begin{eqnarray*}\label{biggest}
r=(a+1)(2d-2+i)+c\leq (a+1)L+I_{d}\leq (a+1)L+L=(a+2)L.
\end{eqnarray*}
By Corollary~\ref{cor-length}, $r\leq (a+2)L<\la$ and
$\nums_{d,i}\subseteq \nums\cap \llbracket 0,\la)\!)$. By
Theorem~\ref{theolength}, $\setl(r,\nums)=\{\ell_r\}$, so
$r\in\nums^{\ell_r}$. By Theorem~\ref{theo-dim}, if
$r=(a+1)(2d-2+i)+c$, with $d\geq 1$, $i\geq 0$ and $c\in\Gamma_i$,
then $\delta_r=d$, $i=\iota_r$, $c=c_r$ and $\ell_r=2d-2+i$. Hence
$r\in \nums_d\cap\nums^{2d-2+i}$ and $\nums_{d,i}\subseteq
\nums_d\cap\nums^{2d-2+i}$.

Let $r\in\nums_d\cap\nums^{2d-2+i}$, where $i\in\llbracket
0,I_d\rrbracket$, so $2d-2+i\leq L$. By Corollary~\ref{cor-length},
$\nums^{2d-2+i}\subseteq\nums\cap \llbracket 0,\la)\!)$ and, by
Theorem~\ref{theolength}, $\setl(r,\nums)=\{\ell_r\}$. Since
$r\in\nums^{2d-2+i}$, it follows that $\ell_r=2d-2+i$.  Moreover,
$r\in\nums_d$ implies $\delta_r=d$. By Theorems~\ref{theo-dim} and
\ref{theo-length}, $r=(a+1)(2\delta_r-2+\iota_r)+c_r$, where
$2\delta_r-2+\iota_r=\ell_r=2d-2+i$. Thus, $\iota_r=i$ and
$r=(a+1)(2d-2+i)+c_r$, with $c_r\in\Gamma_{\iota_r}=\Gamma_i$. We
conclude that $r\in\nums_{d,i}$, which proves item $(2)$.

Now, let us prove item $(3)$. Suppose that $0\leq \ell<2d-2$, in
particular, $\ell<2D-2\leq L$.  By Corollary~\ref{cor-length},
$\nums^\ell\subseteq\nums\cap\llbracket 0,\la)\!)$. If
$r\in\nums^{\ell}$, then, by Theorem~\ref{theolength},
$\setl(r,\nums)=\{\ell_r\}=\{\ell\}$ and, by Theorem~\ref{theo-dim},
$r=(a+1)\ell_r+c$, where $2\delta_r-2\leq
2\delta_r-2+\iota_r=\ell_r=\ell<2d-2$, so $\delta_r<d$ and
$r\not\in\nums_d$. Therefore, $\nums^{\ell}\cap\nums_d=\emptyset$,
whenever $0\leq \ell<2d-2$. It follows that
\[
\bigcup_{i=0}^{I_d}\nums_{d,i}=\bigcup_{i=0}^{I_d}\nums_{d}\cap
\nums^{2d-2+i}=\nums_d\cap\left(\bigcup_{i=0}^{I_d}\nums^{2d-2+i}\right)=
\nums_d\cap\left(\nums^{2d-2}\cup\ldots\cup\nums^{L}\right)=
\nums_d\cap\left(\bigcup_{\ell=0}^{L}\nums^\ell\right),
\]
which proves $(3)$. 

Let $r\in\nums\cap \llbracket 0,(a+2)L\rrbracket$. By
Corollary~\ref{cor-length}, $\nums\cap \llbracket
0,(a+2)L\rrbracket=\bigcup_{\ell=0}^L\nums^{\ell}$ and, in particular,
$r<\la$. So, $r\in\nums^{\ell}$, for some $0\leq\ell\leq L$, and
$\ell_r=\ell\leq L$. By Theorem~\ref{theo-dim}, $r$ can be written as
$r=(a+1)(2d-2+i)+c$, where $d=\delta_r\geq 1$, $i=\iota_r\in\mbn$,
$c=c_r\in\Gamma_i$ and $2d-2+i=\ell_r$. Therefore, $2d-2\leq
2d-2+i=\ell_r\leq L$ and $d\leq (L+2)/2$, so $d\leq \lfloor
(L+2)/2\rfloor = \lfloor (1/2)\lfloor (a+3)/2\rfloor \rfloor =\lfloor
(a+3)/4\rfloor=D$. Hence, $r\in\bigcup_{d=1}^D\nums_d$ and $\nums\cap
\llbracket 0,(a+2)L\rrbracket\subseteq\bigcup_{d=1}^D\nums_d$. It
follows that
\begin{eqnarray*}
\nums\cap \llbracket 0,(a+2)L\rrbracket=
\left(\bigcup_{d=1}^D\nums_d\right)\cap\left(\nums\cap \llbracket
0,(a+2)L\rrbracket\right)=
\left(\bigcup_{d=1}^D\nums_d\right)\cap\left(\bigcup_{\ell=0}^L\nums^{\ell}\right).
\end{eqnarray*}
This proves item $(4)$.

Set $r_d:=(a+1)(2d-2)$. Clearly, $\nums_{d,0}=\{r_d\}$. By
Corollary~\ref{cor-length}, $\nums^0,\ldots,\nums^{L}$ are pairwise
disjoint, where $\nums^{\ell}=\llbracket
a\ell,(a+2)\ell\rrbracket$. If $\ell<L=\lfloor (a-1)/2\rfloor\leq
a/2$, then $(a+2)\ell<a(\ell+1)$. Using that
$\nums_d\cap\nums^{\ell}=\emptyset$, for $0<\ell<2d-2$, and that
$\nums_{d,i}=\nums_d\cap \nums^{2d-2+i}$, it follows that
\begin{eqnarray*}
\min \nums_d\cap\left(\bigcup_{\ell=0}^L\nums^{\ell}\right)= \min
\bigcup_{\ell=2d-2}^L\left(\nums_d\cap
\nums^{\ell}\right)=\min\nums_d\cap \nums^{2d-2}= \min\nums_{d,0}=r_d.
\end{eqnarray*} 
By Theorem~\ref{theo-dim}, $\ell_r=2d-2$ and $\rem_r=c_r+\ell_r=2d-2$.
So $\phi_{r_d}=(d-1,0,d-1)$. By Theorem~\ref{theo-length},
$\fac(r_d,\nums)=\{(d-1,0,d-1)+j\omega : 0\leq j\leq d-1\}$ and
$W_{r_d}=\se x^{d-1}z^{d-1},\ldots,y^{2d-2}\sd$, which proves item
$(5)$.

Suppose that $r\in\nums_{d,1}$. If $c_r=-1$, then $r=(a+1)(2d-1)-1$,
$\ell_r=2d-1$ and $\rem_r=c_r+\ell_r=2d-2$. Therefore,
$\phi_{r}=(d,0,d-1)$, $\fac(r,\nums)=\{(d,0,d-1)+j\omega : 0\leq j\leq
d-1\}$ and $W_{r}=xW_{r_d}$.

If $c_r=0$, then $r=(a+1)(2d-1)$, $\ell_r=2d-1$ and
$\rem_r=c_r+\ell_r=2d-1$. Thus $\phi_{r}=(d-1,1,d-1)$,
$\fac(r,\nums)=\{(d-1,1,d-1)+j\omega : 0\leq j\leq d-1\}$ and
$W_{r}=yW_{r_d}$.

If $c_r=1$, then $r=(a+1)(2d-1)+1$, $\ell_r=2d-1$ and
$\rem_r=c_r+\ell_r=2d$; consequently $\phi_{r}=(d-1,0,d)$,
$\fac(r,\nums)=\{(d-1,0,d)+j\omega :  0\leq j\leq d-1\}$ and
$W_{r}=zW_{r_d}$. This proves item $(6)$.

Suppose that $r\in\nums_{d,i}$, for $2\leq i\leq I_{d}$. If $c_r=-i$,
then $r=(a+1)(2d-2+i)-i$, $\ell_r=2d-2+i$ and
$\rem_r=c_r+\ell_r=2d-2$. Hence,
$\fac(r,\nums)=\{(d-1+i,0,d-1)+j\omega : 0\leq j\leq d-1\}$ and
$W_{r}=x^iW_{r_d}$.

If $c_r=-i+1$, then $r=(a+1)(2d-2+i)-i+1$, $\ell_r=2d-2+i$ and
$\rem_r=c_r+\ell_r=2d-1$. In this setting,
$\fac(r,\nums)=\{(d-2+i,1,d-1)+j\omega :  0\leq j\leq d-1\}$ and
$W_{r}=x^{i-1}yW_{r_d}$.

If $c_r=i-1$, then $r=(a+1)(2d-2+i)+i-1$, $\ell_r=2d-2+i$ and
$\rem_r=c_r+\ell_r=2d-3+2i$. Hence,
$\fac(r,\nums)=\{(d-1,1,d-2+i)+j\omega :  0\leq j\leq d-1\}$ and
$W_{r}=yz^{i-1}W_{r_d}$.

If $c_r=i$, then $r=(a+1)(2d-2+i)+i$, $\ell_r=2d-2+i$ and
$\rem_r=c_r+\ell_r=2d-2+2i$. In this case, 
$\fac(r,\nums)=\{(d-1,0,d-1+i)+j\omega :  0\leq j\leq d-1\}$ and
$W_{r}=z^iW_{r_d}$, which concludes the proof.
\end{proof}

We finish the section with some examples. 

\begin{example}
Let $a\in\mbn$, $a\geq 3$, $L=\lfloor (a-1)/2\rfloor$, $D=\lfloor
(a+3)/4\rfloor$ and $\nums=\se a,a+1,a+2\sd$. In Figures~\ref{fig:a10}
and \ref{fig:a15} we display the triples $(r,\iota_r,c_r)$ of elements
$r$ in $\nums\cap\llbracket 0,(a+2)L\rrbracket$, for $a=10$ and for
$a=15$, respectively, organized in an $L\times D$ table. Concretely,
for an $\ell\in\mbn$, with $0\leq \ell\leq L$, the subset $S^{\ell}$
of elements $r$ of $\nums$ whose factorizations have length $\ell$ are
represented in the row labeled by $\ell$. For a $d\in\mbn$, $1\leq
d\leq D$, the subset $\nums_d$ of elements $r$ of $\nums$ with
denumerant $d$ are represented in the column $d$. The intersection of
the row labeled by $\ell$ and the column labeled by $d$ is precisely
the subset $\nums_{d,i}$, where $i=\ell+2-2d$.

When $a=10$, then $L=\left\lfloor \frac{a-1}{2}\right\rfloor =4$ and
$D=\left\lfloor \frac{a+3}{4}\right\rfloor =3$; the smallest
$r\in\nums$ with two factorizations of different lengths is $60=5\cdot
12=6\cdot 10$. When $a=15$, then $L=\left\lfloor
\frac{a-1}{2}\right\rfloor =7$ and $D=\left\lfloor
\frac{a+3}{4}\right\rfloor =4$; the smallest $r\in\nums$ with two
factorizations of different lengths is $135=16+7\cdot 17=9\cdot
15$. See Lemma~\ref{lemma-length}, $(8)$ and $(11)$. Both numbers, 60
and 135, would appear in the subsequent row.

\begin{figure}
    \centering
{\scriptsize    $\begin{array}{|c|c|c|c|}\hline
&d=1&d=2&d=3
\\ &
\begin{array}{crr}r&\iota_r&c_r\end{array}&
\begin{array}{crr}r&\iota_r&c_r\end{array}&
\begin{array}{crr}r&\iota_r&c_r\end{array}
\\\hline \ell=0&
\begin{array}{crr}0&0&0\end{array}&&
\\\hline \ell=1&
\begin{array}{crr}
{\cred 10}&{\cred 1}&{\cred \text{-}1}\\
{\cgreen 11}&{\cgreen 1}&{\cgreen 0}
\\{\cblue 12}&{\cblue 1}&{\cblue 1}\end{array}&&
\\\hline \ell=2&
\begin{array}{crr}
{\corange 20}&{\corange 2}&{\corange\text{-}2}\\
{\cteal 21}&{\cteal 2}&{\cteal\text{-}1}\\
{\colive 23}&{\colive 2}&{\colive 1}\\
{\cviolet 24}&{\cviolet 2}&{\cviolet 2}\end{array}&
\begin{array}{crr}22&0&0\end{array}&
\\\hline \ell=3&
\begin{array}{crr}
{\corange 30}&{\corange 3}&{\corange \text{-}3}\\
{\cteal 31}&{\cteal 3}&{\cteal \text{-}2}\\
{\colive 35}&{\colive 3}&{\colive 2}\\
{\cviolet 36}&{\cviolet 3}&{\cviolet 3}
\end{array} &
\begin{array}{crr}
{\cred  32}&{\cred 1}&{\cred \text{-}1}\\
{\cgreen 33}&{\cgreen 1}&{\cgreen 0}\\
{\cblue 34}&{\cblue 1}&{\cblue 1}\end{array}&
\\\hline \ell=4&
\begin{array}{crr}
{\corange 40}&{\corange 4}&{\corange \text{-}4}\\
{\cteal 41}&{\cteal 4}&{\cteal \text{-}3}\\
{\colive 47}&{\colive 4}&{\colive 3}\\
{\cviolet 48}&{\cviolet 4}&{\cviolet 4}\end{array} &
\begin{array}{lrr}
{\corange 42}&{\corange 2}&{\corange \text{-}2}\\
{\cteal 43}&{\cteal 2}&{\cteal \text{-}1}\\
{\colive 45}&{\colive 2}&{\colive 1}\\
{\cviolet 46}&{\cviolet 2}&{\cviolet 2}\end{array} & 
\begin{array}{crr}44&0&0\end{array}
\\\hline 
\end{array}$}
\caption{The set of triples $(r,\iota_r,c_r)$, for $r\in\nums\cap
  \llbracket 0,(a+2)L\rrbracket$, when $a=10$. Thus, $\nums=\se
  10,11,12\sd$, $L=\left\lfloor \frac{a-1}{2}\right\rfloor =4$,
  $D=\left\lfloor \frac{a+3}{4}\right\rfloor =3$ and $\llbracket
  0,(a+2)L\rrbracket=\llbracket 0,48\rrbracket$.}
\label{fig:a10}
\end{figure}

\begin{figure}
\centering {\scriptsize $\begin{array}{|c|c|c|c|c|}\hline
    &d=1&d=2&d=3&d=4 \\&
\begin{array}{crr}r&\iota_r&c_r\end{array}&
\begin{array}{crr}r&\iota_r&c_r\end{array}&
\begin{array}{crr}r&\iota_r&c_r\end{array}&
\begin{array}{crr}r&\iota_r&c_r\end{array}    
\\\hline \ell=0&
\begin{array}{crr}0&0&0\end{array}&&&
\\\hline \ell=1&
\begin{array}{crr}
{\cred 15}&{\cred 1}&{\cred \text{-}1}\\
{\cgreen 16}&{\cgreen 1}&{\cgreen 0}
\\{\cblue 17}&{\cblue 1}&{\cblue 1}\end{array}&&&
\\\hline \ell=2&
\begin{array}{crr}
{\corange 30}&{\corange 2}&{\corange\text{-}2}\\
{\cteal 31}&{\cteal 2}&{\cteal\text{-}1}\\
{\colive 33}&{\colive 2}&{\colive 1}\\
{\cviolet 34}&{\cviolet 2}&{\cviolet 2}\end{array}&
\begin{array}{crr}32&0&0\end{array}&&
\\\hline \ell=3&
\begin{array}{crr}
{\corange 45}&{\corange 3}&{\corange \text{-}3}\\
{\cteal 46}&{\cteal 3}&{\cteal \text{-}2}\\
{\colive 50}&{\colive 3}&{\colive 2}\\
{\cviolet 51}&{\cviolet 3}&{\cviolet 3}\end{array} &
\begin{array}{crr}
{\cred  47}&{\cred 1}&{\cred \text{-}1}\\
{\cgreen 48}&{\cgreen 1}&{\cgreen 0}\\
{\cblue 49}&{\cblue 1}&{\cblue 1}\end{array}&&
\\\hline \ell=4&
\begin{array}{crr}
{\corange 60}&{\corange 4}&{\corange \text{-}4}\\
{\cteal 61}&{\cteal 4}&{\cteal \text{-}3}\\
{\colive 67}&{\colive 4}&{\colive 3}\\
{\cviolet 68}&{\cviolet 4}&{\cviolet 4}\end{array} &
\begin{array}{crr}
{\corange 62}&{\corange 2}&{\corange \text{-}2}\\
{\cteal 63}&{\cteal 2}&{\cteal \text{-}1}\\
{\colive 65}&{\colive 2}&{\colive 1}\\
{\cviolet 66}&{\cviolet 2}&{\cviolet 2}\end{array} & 
\begin{array}{crr}64&0&0\end{array}&
\\\hline \ell=5&
\begin{array}{crr}
{\corange 75}&{\corange 5}&{\corange \text{-}5}\\
{\cteal 76}&{\cteal 5}&{\cteal \text{-}4}\\
{\colive 84}&{\colive 5}&{\colive 4}\\
{\cviolet 85}&{\cviolet 5}&{\cviolet 5}\end{array} &
\begin{array}{crr}
{\corange 77}&{\corange 3}&{\corange \text{-}3}\\
{\cteal 78}&{\cteal 3}&{\cteal \text{-}2}\\
{\colive 82}&{\colive 3}&{\colive 2}\\
{\cviolet 83}&{\cviolet 3}&{\cviolet 3}\end{array} &
\begin{array}{crr}
{\cred  79}&{\cred 1}&{\cred \text{-}1}\\
{\cgreen 80}&{\cgreen 1}&{\cgreen 0}\\
{\cblue 81}&{\cblue 1}&{\cblue 1}\end{array}&
\\\hline \ell=6&
\begin{array}{crr}
{\corange 90}&{\corange 6}&{\corange \text{-}6}\\
{\cteal 91}&{\cteal 6}&{\cteal \text{-}5}\\
{\colive 101}&{\colive 6}&{\colive 5}\\
{\cviolet 102}&{\cviolet 6}&{\cviolet 6}\end{array} &
\begin{array}{crr}
{\corange 92}&{\corange 4}&{\corange \text{-}4}\\
{\cteal 93}&{\cteal 4}&{\cteal \text{-}3}\\
{\colive 99}&{\colive 4}&{\colive 3}\\
{\cviolet 100}&{\cviolet 4}&{\cviolet 4}\end{array} &
\begin{array}{crr}
{\corange 94}&{\corange 2}&{\corange \text{-}2}\\
{\cteal 95}&{\cteal 2}&{\cteal \text{-}1}\\
{\colive 97}&{\colive 2}&{\colive 1}\\
{\cviolet 98}&{\cviolet 2}&{\cviolet 2}\end{array}&
\begin{array}{crr}96&0&0\end{array}
\\\hline \ell=7&
\begin{array}{crr}
{\corange 105}&{\corange 7}&{\corange \text{-}7}\\
{\cteal 106}&{\cteal 7}&{\cteal \text{-}6}\\
{\colive 118}&{\colive 7}&{\colive 6}\\
{\cviolet 119}&{\cviolet 7}&{\cviolet 7}\end{array} &
\begin{array}{crr}
{\corange 107}&{\corange 5}&{\corange \text{-}5}\\
{\cteal 108}&{\cteal 5}&{\cteal \text{-}4}\\
{\colive 116}&{\colive 5}&{\colive 4}\\
{\cviolet 117}&{\cviolet 5}&{\cviolet 5}\end{array} &
\begin{array}{crr}
{\corange 109}&{\corange 3}&{\corange \text{-}3}\\
{\cteal 110}&{\cteal 3}&{\cteal \text{-}2}\\
{\colive 114}&{\colive 3}&{\colive 2}\\
{\cviolet 115}&{\cviolet 3}&{\cviolet 3}\end{array}&
\begin{array}{crr}
{\cred 111}&{\cred 1}&{\cred \text{-}1}\\
{\cgreen 112}&{\cgreen 1}&{\cgreen 0}\\
{\cblue 113}&{\cblue 1}&{\cblue 1}\end{array}
\\\hline
\end{array}$}
\caption{The set of triples $(r,\iota_r,c_r)$, for $r\in\nums\cap
  \llbracket 0,(a+2)L\rrbracket$, when $a=15$. Thus, $\nums=\se
  15,16,17\sd$, $L=\left\lfloor\frac{a-1}{2}\right\rfloor=7$,
  $D=\left\lfloor\frac{a+3}{4}\right\rfloor=4$ and $\llbracket
  0,(a+2)L\rrbracket=\llbracket 0,119\rrbracket$.}
\label{fig:a15}
\end{figure}
\end{example}

\begin{remark}
The triple $(r,\iota_r,c_r)$ is written in a specific colour according
to which of the eight sets contains the pair $(\iota_r,c_r)$:
\begin{eqnarray*}
&&\{(0,0)\},{\cred\{(1,-1)\}},{\cgreen\{(1,0)\}},{\cblue\{(1,1)\}},\\&&
{\corange\{(i,-i) :  2\leq i\leq I_d\}},
{\cteal\{(i,-i+1) :  2\leq i\leq I_d\}},
{\colive\{(i,i-1) :  2\leq i\leq I_d\}},
{\cviolet\{(i,i) :  2\leq i\leq I_d\}}.
\end{eqnarray*}
In other words, 
$\nums^{1}$ can be partitioned into three subsets,
$\nums^{2}$ into five,
$\nums^{3}$ into seven,
$\nums^{4}$ into five,
$\nums^{5}$ into seven, and so on. 
\end{remark}

\begin{example}
Let $a\in\mbn$, $a\geq 3$, and $\nums=\se a,a+1,a+2\sd$.
Figure~\ref{fig:mons} shows the $7\times 4$ array of triples
$(W_r,\iota_r,c_r)$, where $r\in \nums$, $\ell_r\in\llbracket
0,7\rrbracket$ and $\delta_r\llbracket 0,4\rrbracket$.

By Corollary~\ref{cor-dim}, the table that displays the elements of
$\nums\cap \llbracket 0,(a+2)L\rrbracket$ has $L=\lfloor
(a-1)/2\rfloor$ rows and $D=\lfloor (a+3)/4\rfloor$ columns. For
instance, for $a\in\{3,4\}$ we must consider the $1\times 1$ top-left
sub-table; for $a\in\{5,6\}$, the $2\times 2$ top-left sub-table; for
$a\in\{7,8\}$, the $3\times 2$ top-left sub-table, and so on. For
$a=10$, we must consider the $4\times 3$ top-left sub-table and, for
$a=15$, the $7\times 4$ top-left sub-table.

Let $(d,i,c)\in\mbz^3$ be a triple, with $1\leq d\leq D$, $0\leq i\leq
I_{d}:=L+2-2d$ and $c\in\Gamma_i$ and let
$r=(a+1)(2d-2+i)+c\in\nums_{d,i}$.  By Corollary~\ref{cor-dim},
$(5)-(7)$, $W_r$ is completely determined in terms of the triple
$(d,i,c)$. Take, for instance, the triple $(d,i,c)=(2,2,-1)$. If
$a=10$, then $L=\lfloor (a-1)/2\rfloor=4$ and $D=\lfloor
(a+3)/4\rfloor=3$, $r=(a+1)(2d-2+i)+c=43$ and $W_{43}=\se
x^2yz,xy^3\sd$; if $a=15$, then $L=\lfloor (b-1)/2\rfloor=7$ and
$D=\lfloor (a+3)/4\rfloor=4$, $r=(a+1)(2d-2+i)+c=63$ and $W_{63}=\se
x^2yz,xy^3\sd$, as well.

\begin{figure}
    \centering
{\scriptsize
$\begin{array}{|c|c|c|c|c|}\hline
&d=1&d=2&d=3&d=4
\\&
\begin{array}{crr}W_r&\iota_r&c_r\end{array}&
\begin{array}{crr}W_r&\iota_r&c_r\end{array}&
\begin{array}{crr}W_r&\iota_r&c_r\end{array}&
\begin{array}{crr}W_r&\iota_r&c_r\end{array}    
\\\hline \ell=0&
\begin{array}{crr}1&0&0\end{array}&&&
\\\hline \ell=1&
\begin{array}{crr}
{\cred x}&{\cred 1}&{\cred \text{-}1}\\
{\cgreen y}&{\cgreen 1}&{\cgreen 0}
\\{\cblue z}&{\cblue 1}&{\cblue 1}\end{array}&&&
\\\hline \ell=2&
\begin{array}{crr}
{\corange x^2}&{\corange 2}&{\corange\text{-}2}\\
{\cteal xy}&{\cteal 2}&{\cteal\text{-}1}\\
{\colive yz}&{\colive 2}&{\colive 1}\\
{\cviolet z^2}&{\cviolet 2}&{\cviolet 2}\end{array}&
\begin{array}{crr}xz,y^2&0&0\end{array}&&
\\\hline \ell=3&
\begin{array}{crr}
{\corange x^3}&{\corange 3}&{\corange \text{-}3}\\
{\cteal x^2y}&{\cteal 3}&{\cteal \text{-}2}\\
{\colive yz^2}&{\colive 3}&{\colive 2}\\
{\cviolet z^3}&{\cviolet 3}&{\cviolet 3}\end{array} &
\begin{array}{crr}
{\cred  x^2z,xy^2}&{\cred 1}&{\cred \text{-}1}\\
{\cgreen xyz,y^3}&{\cgreen 1}&{\cgreen 0}\\
{\cblue xz^2,y^2z}&{\cblue 1}&{\cblue 1}\end{array}&&
\\\hline \ell=4&
\begin{array}{crr}
{\corange x^4}&{\corange 4}&{\corange \text{-}4}\\
{\cteal x^3y}&{\cteal 4}&{\cteal \text{-}3}\\
{\colive yz^3}&{\colive 4}&{\colive 3}\\
{\cviolet z^4}&{\cviolet 4}&{\cviolet 4}\end{array} &
\begin{array}{crr}
{\corange x^3z,x^2y^2}&{\corange 2}&{\corange \text{-}2}\\
{\cteal x^2yz,xy^3}&{\cteal 2}&{\cteal \text{-}1}\\
{\colive xyz^2,y^3z}&{\colive 2}&{\colive 1}\\
{\cviolet xz^3,y^2z^2}&{\cviolet 2}&{\cviolet 2}\end{array} & 
\begin{array}{crr}x^2z^2,xy^2z,y^4&0&0\end{array}&
\\\hline \ell=5&
\begin{array}{crr}
{\corange x^5}&{\corange 5}&{\corange \text{-}5}\\
{\cteal x^4y}&{\cteal 5}&{\cteal \text{-}4}\\
{\colive yz^4}&{\colive 5}&{\colive 4}\\
{\cviolet z^5}&{\cviolet 5}&{\cviolet 5}\end{array} &
\begin{array}{crr}
{\corange x^4z,x^3y^2}&{\corange 3}&{\corange \text{-}3}\\
{\cteal x^3yz,x^2y^3}&{\cteal 3}&{\cteal \text{-}2}\\
{\colive xyz^3,y^3z^2}&{\colive 3}&{\colive 2}\\
{\cviolet xz^4,y^2z^3}&{\cviolet 3}&{\cviolet 3}\end{array} &
\begin{array}{crr}
{\cred  x^3z^2,x^2y^2z,xy^4}&{\cred 1}&{\cred \text{-}1}\\
{\cgreen x^2yz^2,xy^3z,y^5}&{\cgreen 1}&{\cgreen 0}\\
{\cblue x^2z^3,xy^2z^2,y^4z}&{\cblue 1}&{\cblue 1}\end{array}&
\\\hline \ell=6&
\begin{array}{crr}
{\corange x^6}&{\corange 6}&{\corange \text{-}6}\\
{\cteal x^5y}&{\cteal 6}&{\cteal \text{-}5}\\
{\colive yz^5}&{\colive 6}&{\colive 5}\\
{\cviolet z^6}&{\cviolet 6}&{\cviolet 6}\end{array} &
\begin{array}{crr}
{\corange x^5z,x^4y^2}&{\corange 4}&{\corange \text{-}4}\\
{\cteal x^4yz,x^3y^3}&{\cteal 4}&{\cteal \text{-}3}\\
{\colive xyz^4,y^3z^3}&{\colive 4}&{\colive 3}\\
{\cviolet xz^5,y^2z^4}&{\cviolet 4}&{\cviolet 4}\end{array} &
\begin{array}{crr}
{\corange x^4z^2,x^3y^2z,x^2y^4}&{\corange 2}&{\corange \text{-}2}\\
{\cteal x^3yz^2,x^2y^3z,xy^5}&{\cteal 2}&{\cteal \text{-}1}\\
{\colive x^2yz^3,xy^3z^2,y^5z}&{\colive 2}&{\colive 1}\\
{\cviolet x^2z^4,xy^2z^3,y^4z^2}&{\cviolet 2}&{\cviolet 2}\end{array}&
\begin{array}{crr}x^3z^3,x^2y^2z^2,xy^4z,y^6&0&0\end{array}
\\\hline \ell=7&
\begin{array}{crr}
{\corange x^7}&{\corange 7}&{\corange \text{-}7}\\
{\cteal x^6y}&{\cteal 7}&{\cteal \text{-}6}\\
{\colive yz^6}&{\colive 7}&{\colive 6}\\
{\cviolet z^7}&{\cviolet 7}&{\cviolet 7}\end{array} &
\begin{array}{crr}
{\corange x^6z,x^5y^2}&{\corange 5}&{\corange \text{-}5}\\
{\cteal x^5yz,x^4y^3}&{\cteal 5}&{\cteal \text{-}4}\\
{\colive xyz^5,y^3z^4}&{\colive 5}&{\colive 4}\\
{\cviolet xz^6,y^2z^5}&{\cviolet 5}&{\cviolet 5}\end{array} &
\begin{array}{crr}
{\corange x^5z^2,x^4y^2z,x^3y^4}&{\corange 3}&{\corange \text{-}3}\\
{\cteal x^4yz^2,x^3y^3z,x^2y^5}&{\cteal 3}&{\cteal \text{-}2}\\
{\colive x^2yz^4,xy^3z^3,y^5z^2}&{\colive 3}&{\colive 2}\\
{\cviolet x^2z^5,xy^2z^4,y^4z^3}&{\cviolet 3}&{\cviolet 3}\end{array}&
\begin{array}{crr}
{\cred  x^4z^3,x^3y^2z^2,x^2y^4z,xy^6}&{\cred 1}&{\cred \text{-}1}\\
{\cgreen x^3yz^3,x^2y^3z^2,xy^5z,y^7}&{\cgreen 1}&{\cgreen 0}\\
{\cblue x^3z^4,x^2y^2z^3,xy^4z^2,y^6z}&{\cblue 1}&{\cblue 1}\end{array}
\\\hline
\end{array}$}
\caption{The set of triples $(W_r,\iota_r,c_r)$, for $r\in\nums$, with
$\ell_r\in\llbracket 0,7\rrbracket$ 
and $\delta_r\in\llbracket 0,4\rrbracket$.}
\label{fig:mons}
\end{figure}
\end{example}

\section{A Betti-element perspective}\label{sec-Betti}

Let $\numGen$ be a numerical semigroup minimally generated by
$\{n_1,\dots,n_e\}$, that is, $\{n_1,\dots,n_e\}=\numGen^*\setminus
(\numGen^*+\numGen^*)$, with $\numGen^*=\numGen\setminus\{0\}$. The
map
\begin{equation}\label{eq:fact-hom}
\varphi: \mathbb{N}^e\to\numGen, (a_1,\dots,a_e)\mapsto
a_1n_1+\dots+a_en_e
\end{equation}
is a surjective monoid homomorphism, known as the \emph{factorization
  homomorphism} of $\numGen$, and consequently $\numGen$ is isomorphic
to $\mathbb{N}^e/\ker(\varphi)$, where $\ker(\varphi)=\{(a,b)\in
\mathbb{N}^e\times \mathbb{N}^e : \varphi(a)=\varphi(b)\}$ is the
\emph{kernel congruence} of $\varphi$. Clearly,
$\fac(r,\numGen)=\varphi^{-1}(r)$.

A \emph{presentation} of $\numGen$ is a generating set (as a
congruence) of $\ker(\varphi)$. A presentation is \emph{minimal} if
none of its proper subsets generates $\ker(\varphi)$.

Given $r\in\numGen$, let $\nabla_r$ be the graph whose set of vertices
is $\fac(r,\numGen)$, the set of factorizations of $r$ in $\numGen$,
and two vertices are joined by an edge if they have common support
(that is, the dot product of these two factorizations is not zero). We
say that $r$ is a \emph{Betti element} (or Betti degree) of $\numGen$
if $\nabla_r$ is not connected. The set of Betti elements of a
numerical semigroup is finite (see for instance
\cite[Proposition~66]{ns-ap}). By the construction explained after
\cite[Theorem~10]{ns-ap}, every (minimal) presentation of $\numGen$
can be obtained by taking pairs of factorizations of Betti elements.

Let $\betti(\numGen)$ denote the set of Betti elements of
$\numGen$. Set $\bbetti(\numGen)=\betti(\numGen)\cap \ulf(\numGen)$
and set $\ubetti(\numGen)= \betti(\numGen)\setminus \bbetti(\numGen)$
(B and U standing for \emph{balanced} and \emph{unbalanced}, following
the idea of unbalanced relations in \cite{l-f}).

Recall that for $r\in \numGen\setminus\{0\}$, the \emph{Ap\'ery set}
of $r$ in $\numGen$ is $\operatorname{Ap}(\numGen,r)=\{s\in \numGen :
s-r\not\in \numGen\}$. For $X \subseteq \numGen\setminus\{0\}$, set
$\operatorname{Ap}(\numGen,X)=\bigcap_{x\in
  X}\operatorname{Ap}(\numGen,x)=\numGen\setminus (X+\numGen)$. Notice
that we are allowing $X$ to be empty, in which case
$\operatorname{Ap}(\numGen,\emptyset)=\numGen$.

\begin{theorem}\label{th-appendix}
Let $\numGen$ be a numerical semigroup. Then, 
\begin{eqnarray*}
 \ulf(\numGen)=\operatorname{Ap}(\numGen,\ubetti(\numGen))\end{eqnarray*}
\end{theorem}
\begin{proof}
Let $r\in\numGen$. If $r\not\in
\operatorname{Ap}(\numGen,\ubetti(\numGen))$,
then there exists $b\in \ubetti(\numGen)$ such that
$r-b\in \numGen$. As $b\in\ubetti(\numGen)$,
there exists $z,z'\in \fac(b,\numGen)$ such that
$|z|\neq |z'|$. Let $s=r-b\in\numGen$, and take $x\in
\fac(s,\numGen)$. Then,
$r=b+s=\varphi(z)+\varphi(x)=\varphi(z+x)=\varphi(z')+\varphi(x)
=\varphi(z'+x)$, meaning that $z+x,z'+x\in
\fac(r,\numGen)$. As $|z]\neq |z'|$, we deduce that
  $|z+x|=|z|+|x|\neq |z'|+|x|=|z'+x|$, which proves that $r$ has two
  factorizations of different length.

Now, suppose that $r\in\operatorname{Ap}(\numGen,\ubetti(\numGen))$.
If $r$ has only one factorization, then there is nothing to prove (as
a matter of fact, $r\in \operatorname{Ap}(\numGen,\betti(\numGen))$ by
\cite[Corollary~3.8]{isolated}). Let $z$ and $z'$ be two
factorizations of $r$. In light of \cite[Proposition~65]{ns-ap} and
the construction explained right after the proof of
\cite[Theorem~10]{ns-ap}, there exists a sequence $z_1,\dots,z_n\in
\mathbb{N}^e$ such that $z_1=z$, $z_n=z'$ and for all $i\in
\{1,\dots,n-1\}$, $(z_i,z_{i+1})=(x_i+y_i,x_{i+1}+y_i)$ for some
$y_i\in\mathbb{N}^e$ and either $(x_i,x_{i+1})$ or $(x_{i+1},x_i)$ in
a presentation of $\numGen$. In particular, for each $i$, there exists
$b_i\in \betti(\numGen)$ such that $x_i,x_{i+1}\in
\fac(b_i,\numGen)$. Notice that as $(z_i,z_{i+1})\in \ker(\varphi)$,
we deduce that $\{z_1,\dots,z_n\}\subseteq \fac(r,\numGen)$. Let
$s_i=\varphi(y_i)\in\numGen$. Then
$r=\varphi(z_i)=\varphi(x_i+y_i)=\varphi(x_i)+\varphi(y_i)=b_i+s_i$. From
$r\in\operatorname{Ap}(\numGen,\ubetti(\numGen))$, we then deduce that
$b_i\not\in \ubetti(\numGen)$ for all $i$, and consequently $b_i\in
\bbetti(\numGen)$ for all $i\in \{1,\dots,n-1\}$. From
$\bbetti(\numGen)\subseteq \ulf(\numGen)$, it follows that
$|x_i|=|x_{i+1}|$, and so $|z_i|=|z_{i+1}|$ for all $i$, which
ultimately proves that $|z|=|z'|$.
\end{proof}

\begin{example}\label{ex:ed-two}
Let $a$ and $b$ be two coprime positive integers greater than one, and
set $\numGen=\langle a,b\rangle$. It is well known that
$\betti(\numGen)=\{ab\}$ and that $\{((b,0),(0,a))\}$ is a minimal
presentation of $\numGen$ (see for instance
\cite[Example~8.22]{rg}). If follows that
$\betti(\numGen)=\ubetti(\numGen)$, and by Theorem~\ref{th-appendix},
the elements in $\operatorname{Ap}(\numGen,ab)$ are the only elements
in $\numGen$ all of whose factorizations have equal length (actually,
each of these elements has a unique factorization according to
\cite[Corollary~3.8]{isolated}). By \cite[Theorem~14]{fr-two}, with
$u=0$ and $v=a$, $\operatorname{Ap}(\numGen,ab)=\{ \alpha a+\beta b :
\alpha\in \{0,\dots,b-1\}, \beta\in\{0,\dots,a-1\}\}$. Observe that
$\numGen$ is \emph{length-factorial} (all the factorizations of an
element in the monoid have different lengths). The minimal
presentation of $\numGen$ is generated by a single relation, and thus
it is cyclic (see \cite[Theorem~3.1]{l-f}).
\end{example}

\begin{example}
Let $\numGen$ be the submonoid of $\mathbb{N}^2$ generated by
$\{(2,0),(1,1),(0,2)\}$. A minimal presentation of $\numGen$ is
$\{((0,2,0),(1,0,1))\}$. Hence, $\betti(\numGen)=\{(2,2)\}=
\bbetti(\numGen)$. Thus, for every element in $\numGen$, the set of
lengths of its factorizations is a singleton. These monoids are known
in the literature as \emph{half-factorial} monoids. Notice that this
monoid is precisely the set of non-negative integer solutions of
$x+y\equiv 0\pmod 2$; the set of minimal generators (atoms) of our
monoid is contained in a hyperplane (see \cite[Proposition~1]{kl}).
\end{example}

\begin{example}
    Let $\nums=\langle 10,11,12\rangle$. By using the
    \texttt{numericalsgps} \cite{numericalsgps} \texttt{GAP}
    \cite{gap} package we can compute the Betti elements of $\nums$
    and the Ap\'ery set corresponding to unbalanced Betti elements.
\begin{verbatim}
gap> s:=Numericalemigroup(10,11,12);;
gap> BettiElements(s);
[ 22, 60 ]
gap> AperyList(s,60);
[ 0, 61, 62, 63, 64, 65, 66, 67, 68, 69, 10, 11, 12, 73, 74, 75, 76, 77, 
  78, 79, 20, 21, 22, 23, 24, 85, 86, 87, 88, 89, 30, 31, 32, 33, 34, 35, 
  36, 97, 98, 99, 40, 41, 42, 43, 44, 45, 46, 47, 48, 109, 50, 51, 52, 53, 
  54, 55, 56, 57, 58, 59 ]    
\end{verbatim}
For $\nums=\langle 15,16,17\rangle$, we obtain:
\begin{verbatim}
gap> s:=Numericalemigroup(15,16,17);;
gap> BettiElements(s);
[ 32, 135, 136 ]
gap> Intersection(AperyList(s,135),AperyList(s,136));
[ 0, 15, 16, 17, 30, 31, 32, 33, 34, 45, 46, 47, 48, 49, 50, 51, 60, 61, 
  62, 63, 64, 65, 66, 67, 68, 75, 76, 77, 78, 79, 80, 81, 82, 83, 84, 85, 
  90, 91, 92, 93, 94, 95, 96, 97, 98, 99, 100, 101, 102, 105, 106, 107, 
  108, 109, 110, 111, 112, 113, 114, 115, 116, 117, 118, 119, 120, 121, 
  122, 123, 124, 125, 126, 127, 128, 129, 130, 131, 132, 133, 134, 137, 
  138, 139, 140, 141, 142, 143, 144, 145, 146, 147, 148, 149, 154, 155, 
  156, 157, 158, 159, 160, 161, 162, 163, 164, 171, 172, 173, 174, 175,
  176, 177, 178, 179, 188, 189, 190, 191, 192, 193, 194, 205, 206, 207, 
  208, 209, 222, 223, 224, 239 ]    
\end{verbatim}
\end{example}

\begin{remark}
Theorem~\ref{th-appendix} can be stated in a more general setting. It
holds true for any commutative, cancellative, reduced monoid
fulfilling the ascending chain condition on principal ideals. These
monoids are atomic, and a minimal presentation is constructed by using
the same idea of choosing pairs of factorizations of elements whose
graphs are non-connected (see \cite{acpi} for more details). In
particular, Theorem~\ref{th-appendix} works for any submonoid of
$\mathbb{N}^{(I)}$, with $I$ a set of non-negative integers.
\end{remark}

\begin{corollary}\label{cor-minbetti}
Let $\numGen$ be a numerical semigroup. Then $\ulf(\numGen)$ is
finite. If $b=\min(\ubetti(\numGen))$, then $\numGen\cap \llbracket
0,b)\!)\subseteq \ulf(\numGen)$ and $b\not\in\ulf(\numGen)$.
\end{corollary}
\begin{proof}
The proof follows easily from Theorem~\ref{th-appendix} and the fact
that the Ap\'ery set of an element $r$ in $\numGen$ has exactly $r$
elements \cite[Lemma~1]{ns-ap}. Notice that $\numGen\cap \llbracket
0,b)\!)\subseteq \operatorname{Ap}(\numGen, \ubetti(\numGen))$.
\end{proof}

\begin{example}\label{ex:arith}
Let $a$ be a positive integer and let $d$ be a positive integer
coprime with $a$. Let $\numGen=\langle a,a+d,\dots, a+nd\rangle$ with
$n\le a-1$. Then, $\{a,a+d,\dots,a+nd\}$ is a minimal generating set
of $\numGen$. Let $a=c n+b$, with $c$ a positive integer and $b\in
\{1,\dots,n\}$. Let $\mathbf{e_i}$ be the $i$th row of the identity
matrix; the map $\varphi$ in \eqref{eq:fact-hom}, maps $\mathbf{e}_i$
to $a+(i-1)d$, $i\in\{1,\dots,n+1\}$. According to
\cite[Theorem~1.1]{gss} (via Herzog's correspondence \cite{herzog}),
we know that
\begin{multline}
\label{eq:min-pres-arith}
\rho=\left\{(\mathbf{e}_i+\mathbf{e}_{j+1},
\mathbf{e}_j+\mathbf{e}_{i+1} ) : i\in \{1,\dots,n-1\}, j\in
\{i+1,\dots,n\}\right\}\\ \cup
\left\{((c+d)\mathbf{e}_1+\mathbf{e}_{k-2},c\,\mathbf{e}_{n+1}+\mathbf{e}_{b+k-2}
): k\in \{3,\dots, n+3-b\}\right\},
\end{multline}
is a minimal presentation of $\numGen$. In particular, 
\begin{multline*}
    \betti(\numGen)=\{(a+id)+(a+(j+1)d) : i\in \{0,\dots,n-2\}, j\in
    \{i+1,\dots,n-1\}\}\\ \cup \{ (c+d)a+(a+(k-3)d) : k\in
    \{3,\dots,n+3-b\}\}.
\end{multline*} 

Observe that $(c+d)a+(a+(k-3)d)$ has two factorizations of different
length, say $(c+d)\mathbf{e}_1+\mathbf{e}_{k-2}$ and
$c\,\mathbf{e}_{n+1}+\mathbf{e}_{b+k-2}$ (the length of the first is
$c+d+1$, while the length of the second is $c+1$).  Hence,
$\{(c+d)a+(a+(k-3)d) : j\in \{3,\dots,n+3-b\}\}\subseteq
\ubetti(\numGen)$.

Now, suppose that $(a+id)+(a+(j+1)d)$, with $1\le i<j\le n$, has a
factorization of length greater than two, that is,
$(a+id)+(a+(j+1)d)\in \ubetti(\numGen)$. Then, there exists a chain of
factorizations $z_1,\dots,z_n$ of $(a+id)+(a+(j+1)d)$ such that
$z_1=\mathbf{e}_{i+1}+\mathbf{e}_{j+2}$, $|z_n|>2$, and
$(z_l,z_{l+1})=(x_l+y_l,x_{l+1}+y_l)$ with either $(x_l,x_{l+1})\in
\rho$ or $(x_{l+1},x_l)\in \rho$, and $y_l\in \mbn^{n+1}$ for every
$l\in \{1,\dots,n-1\}$ (we use the same argument employed in the proof
of Theorem~\ref{th-appendix}). It follows that for some $l$ (we take
$l$ minimum with this condition, we have that $|z_l|=2$ and
$|z_{l+1}|>2$; this is because $|z_1|=2$ and none of the $z_l$ can
have length equal to one (that would translate to
$a+td=(a+id)+(a+(j+1)d)$ for some $t$, contradicting that
$\{a,a+d,\dots,a+nd\}$ is a minimal set of generators of
$\numGen$). As $|z_l|=2$ and $z_l=x_l+y_l$, we have $1<|x_l|\le
|z_l|=2$, and so $|x_l|=2$, which forces $y_l=0$. Hence,
$z_{l+1}=x_{l+1}$ and so $|x_{l+1}|>2$.  The only possibility for this
to happen is that $c=1$, $x_l=z_l=\mathbf{e}_{n+1}+\mathbf{e}_{b+k-2}$
and $x_{l+1}=z_{l+1}=(1+d)\mathbf{e}_1+\mathbf{e}_{k-2}$ for some
$k\in \{3,\dots,n+3-b\}$. As $z_l$ and $z_{l+1}$ are factorizations of
$(a+id)+(a+(j+1)d)$, we deduce that
$(a+id)+(a+(j+1)d)=(1+d)a+(a+(k-3)d)\in \ubetti(\numGen)$ and we can
conclude, no matter the value of $c$ is, that
\[
\ubetti(\langle a,a+d,\dots,a+nd\rangle)=\{(c+d)a+(a+(k-3)d) : k\in
\{3,\dots,n+3-b\}\}.
\]
\end{example}

\section{Specializing in numerical semigroups generated by
  three consecutive integers}

Now, let us focus again on the particular case of $\nums=\langle
a,a+1,a+2\rangle$ for some positive integer $a$. The case $a=1$ is
precisely $\nums=\mathbb{N}$, which is \emph{factorial}, that is, all
its elements have a unique factorization. For $a=2$, we have a
particular instance of Example~\ref{ex:ed-two}, and thus $\nums$ is
\emph{length-factorial}.

For the rest of this section we consider $a\ge 3$. In this setting,
$\{a,a+1,a+2\}$ minimally generates $\nums$. Write $a=2c+b$ with $b\in
\mathbb{N}$ and $b\in \{1,2\}$. A minimal presentation of $\nums$ is
given by \eqref{eq:min-pres-arith} with $d=1$ and $n=2$. Under the
standing hypothesis, the only choice for $i$ and $j$ is $i=0$ and
$j=1$, respectively.

If $b=2$, the only possible choice for $k$ is 3, and consequently a
minimal presentation of $\nums$ is
\[\{((0,2,0),(1,0,1)), ((c+2,0,0),(0,0,c+1))\}.\] 
These pairs of factorizations (also known as relations or
relators) correspond to the following identities:
\begin{itemize}
    \item $2(a+1)=a+(a+2)$, and 
    \item $(c+1)(a+2)=(c+2)a$. 
\end{itemize} 

For $b=1$, we have $k\in\{3,4\}$, and so we obtain that
\[\{((0,2,0),(1,0,1)), ((c+1,1,0),(0,0,c+1)), ((c+2,0,0),(0,1,c))\}\]
is a minimal presentation of $\nums$; which encodes the following identities:
\begin{itemize}
\item $2(a+1)=a+(a+2)$,
\item $(c+1)(a+2)=(c+1)a+(a+1)$, and
\item
  $(c+2)a=(a+1)+c(a+2)$.
\end{itemize}
In particular, $\bbetti(\nums)=\{2a\}$, and 
\[
\ubetti(\nums)=\begin{cases} \{\frac{a}2(a+2)\}, \text{if } b=2,\\ 
\{(\frac{a-1}{2}+2)a,(\frac{a-1}{2}+1)(a+2)\}, \text{if } b=1.
\end{cases}
\]
Observe that for $b=1$,
$(c+2)a=(c+1)a+a<(c+1)a+(a+1)=(c+1)(a+2)$. Therefore,
\begin{itemize}
\item if $a$ is even ($b=2$), then
  $\ulf(\nums)=\operatorname{Ap}(\nums,\frac{a}2(a+2))$ and
  $\mathcal{L}_a=\frac{a}2(a+2)$;
\item if $a$ is odd ($b=1$), then
  $\ulf(\nums)=\operatorname{Ap}(\nums,\{(\frac{a-1}{2}+2)a,
  (\frac{a-1}{2}+1)(a+2)\})$ and $\mathcal{L}_a=(\frac{a-1}{2}+2)a$.
\end{itemize}

Next, we give a precise description of
$\operatorname{Ap}(\nums,\ubetti(\nums))$. We distinguish to cases,
depending on the parity of $a$.

\begin{lemma}\label{lem:ap-a-even}
 Let $\nums=\langle a,a+1,a+2\rangle$, with $a$ an even integer
 greater than two. Then, every $r\in\nums$ can be expressed as
 $r=\lambda a+\mu (a+1)+\eta (a+2)$ with $\lambda\in \mathbb{N}$,
 $\mu\in \{0,1\}$ and $\eta\in \{0,1,\dots,a/2-1\}$. Moreover,
 \[\operatorname{Ap}\left(\nums,
 \ubetti(\nums)\right)=\left\{\lambda a+\mu
 (a+1)+\eta (a+2) : \lambda\in \left\{0,\dots,\frac{a}2\right\},
 \mu\in \{0,1\}, \eta\in
 \left\{0,\dots,\frac{a}2-1\right\}\right\}.\]
\end{lemma}
\begin{proof}
Set $k=a/2$. We already know that a minimal presentation for
$\nums$ is $$\{((0,2,0),(1,0,1)), ((k+1,0,0),(0,0,k))\},$$ and
$\ubetti(\nums)=\{k(a+2)\}$.
    
Every element $r$ of $\nums$ can be expressed as $r=x_1 a + x_2 (a+1)+
x_3 (a+2)$ with $x_1, x_2,x_3 \in \mathbb{N}$. Notice that if $x_2>2$,
then $x_2(a+1)=(x_2-2)(a+1)+2(a+1)=a+(x_2-2)+(a+2)$ (we are using the
first element in the presentation of $\nums$). By applying this
relation as many times as needed, we can write
$r=y_1a+y_2(a+1)+y_3(a+2)$ with $y_1,y_2\in\mathbb{N}$ and $y_2\in
\{0,1\}$. If $y_3>k$, then $y_3(a+2)=(y_3-k)(a+2)+(k+1)a$ (we are
using the second relator of our minimal presentation). It follows
easily that we can write $r=z_1 a+z_2 (a+1)+z_3(a+2)$, with $z_1\in
\mathbb{N}$, $z_2\in\{0,1\}$ and $z_3\in \{0,1,\dots,k-1\}$.
    
If $r\in \operatorname{Ap}\left(\nums, k(a+2)\right)$, then $z_1<k+1$,
since otherwise $r=(z_1-(k+1))a+(k+1)a+z_2(a+1)+z_3(a+2)=(z_1-(k+1))a+
z_2(a+1)+z_3(a+2)+k(a+2)\not\in \operatorname{Ap}\left(\nums,
k(a+2)\right)$. Hence, $r\in\operatorname{Ap}\left(\nums,
k(a+2)\right)$, forces $z_1\in \{0,1,\dots,k\}$. This implies that
$\operatorname{Ap}(\nums,k(a+2))\subseteq \{\lambda a+\mu (a+1)+\eta
(a+2) : \lambda\in \{0,1,\dots,a/2\}, \mu\in \{0,1\}, \eta\in
\{0,1,\dots,a/2-1\}\}$. We also know that the cardinality of
$\operatorname{Ap}(\nums,k(a+2))$ is $k(a+2)$, and the cardinality of
$\{\lambda a+\mu (a+1)+\eta (a+2) : \lambda\in \{0,1,\dots,a/2\},
\mu\in \{0,1\}, \eta\in \{0,1,\dots,a/2-1\}\}$ is at most
$2\frac{a}2(\frac{a}2+1)=\frac{a}2(a+2)=k(a+2)$, proving in this way
that both sets must be the same.
\end{proof}

For $a$ odd, the description of
$\operatorname{Ap}(\nums,\ubetti(\nums))$ is a bit more tricky. We
will make use of the following characterization of membership to
$\nums$, which is a particular instance of \cite[Corollary~3]{gr} (see
also Theorem~\ref{theo-length}): for every $n\in \mathbb{Z}$,
\begin{equation}\tag{M}
\label{eq:membership}
n\in\nums \text{ if and only if } n\bmod a\le 2\left\lfloor
\frac{n}{a}\right\rfloor.
\end{equation}

\begin{lemma}\label{lem:ap-a-odd}
Let $\nums=\langle a,a+1,a+2\rangle$, with $a$ an odd integer
greater than two. Then,
\[
 \operatorname{Ap}\left(\nums, \ubetti(\nums)\right)= \left\{\lambda
 a+\mu(a+1)+\eta (a+2): \begin{array}{l}\lambda\in\left\{0,\dots,
   \frac{a-1}{2}+1-\mu\right\},\\ \mu\in\{0,1\},
   \eta\in\left\{0,\dots,\frac{a-1}{2}-\mu\right\}\end{array}\right\}.
\]
\end{lemma}
\begin{proof}
Set $k=(a-1)/2$. Recall that, in this setting,
$\ubetti(\nums)=
\left\{\left(k+2\right)a,\left(k+1\right)(a+2)\right\}$.  Notice that
$\{\lambda a+\mu(a+1)+\eta (a+2) : \lambda\in \{0,\dots,k+1-\mu\},
\eta\in \{0,\dots,k-\mu\},\mu\in\{0,1\}\}$ equals
    \begin{multline*}
    \{\lambda a+\eta (a+2) : \lambda\in \{0,\dots,k+1\}, \eta\in
    \{0,\dots,k\}\} \\ \cup \{\lambda a+(a+1)+\eta (a+2) :
    \lambda\in \{0,\dots,k\}, \eta\in \{0,\dots,k-1\}\}.
    \end{multline*}
We can argue as in the proof of Lemma~\ref{lem:ap-a-even} to deduce
that if $r\in \operatorname{Ap}(\nums,\{(k+2)a,(k+1)(a+2)\})$,
then $r=\lambda a+\mu(a+1)+\eta(a+2)$ with $\mu\in \{0,1\}$, and if
$\mu=0$, then $\lambda\in \{0,\dots,k+1\}$ and $\eta\in
\{0,\dots,k\}$; while if $\mu=1$, then $\lambda\in \{0,\dots,k\}$ and
$\eta\in \{0,\dots,k-1\}$.

For the other inclusion, we prove that
\begin{itemize}
        \item $\{\lambda a+\eta (a+2) : \lambda\in \{0,\dots,k+1\},
          \eta\in \{0,\dots,k\}\}\subseteq
          \operatorname{Ap}(\nums,\{(k+2)a,(k+1)(a+2)\})$, and
        \item $\{\lambda a+(a+1)+\eta (a+2) : \lambda\in
          \{0,\dots,k\}, \eta\in \{0,\dots,k-1\}\}\subseteq
          \operatorname{Ap}(\nums,\{(k+2)a,(k+1)(a+2)\})$.
\end{itemize}
    
We deal with both inclusions separately. We start with $r=\lambda
a+\eta (a+2)$ with $0\le \lambda\le k+1$ and $0\le \eta\le k$. In
this case, $2\lambda\le 2k+2=a+1$ and $2\eta\le 2k=a-1$.
\begin{itemize}
\item Let $n=r-(k+2)a$. Then, $n=-\nu a+\eta(a+2)$, with
  $\nu\in\{1,\dots,k+2\}$. We have $n\bmod a=2\eta$, while $2\lfloor
  n/a\rfloor= 2\lfloor -\nu+\eta + 2\eta/a\rfloor$. As $2\eta\le
  2k<a$, we get $2\lfloor n/a\rfloor= 2\eta-2\nu < 2\eta$, which
  by \eqref{eq:membership}, forces $n\not\in\nums$.
        
\item Now, set $n=r-(k+1)(a+2)$. Then, $n=\lambda a - \nu(a+2)$, with
  $\nu\in \{1,\dots,k+1\}$. If $\nu=k+1$, then $n=\lambda
  a-(k+1)(a+2)=\lambda a-(a+1)-(k+1)a\le -(a+1)<0$, and so $n\not\in
 \nums$. Thus, we may assume that $\nu\le k$, and so $2\nu\le
  a-1$. Thus, $n\bmod a=(-2\nu)\bmod a=a-2\nu$. Also, $2 \lfloor
  n/a\rfloor= 2\lfloor \lambda -\nu -2\nu/a
  \rfloor=2(\lambda-\nu-1)$. Now, $a-2\nu\le 2\lambda -2\nu-2$ if and
  only if $a\le 2\lambda -2$. We know that $2\lambda -2\le
  2(k+1)-2=2k=a-1$, and so by \eqref{eq:membership}, $n\not\in
 \nums$.     
    
 Now, let $r=\lambda a+(a+1)+\eta (a+2)$ with $0\le \lambda\le k$ and
 $0\le \eta\le k-1$. In particular, $2\lambda\le 2k=a-1$ and $2\eta\le
 2k-2=a-3$.
\end{itemize}

\begin{itemize} 
\item Take $n=r-(k+2)a$. Then, $n=-\nu a+(a+1)+\eta(a+2)$, with
  $\nu\in \{2,\dots,k+2\}$. In this case, $n\bmod a=2\eta+1$ and
  $2\lfloor n/a\rfloor=2\lfloor -\nu+1+\eta+
  (2\eta+1)/a\rfloor=2(\eta-\nu+1)$. Hence, in view of
  \eqref{eq:membership}, $n\in\nums$ if and only if
  $2\eta+1\le 2\eta-2\nu+2$, which is equivalent to $2\nu\le
  1$. Therefore, $n\not\in\nums$.

\item Set $n=r-(k+1)(a+2)$. Then, $n=\lambda a+(a+1)-\nu(a+2)$, with
  $\nu\in \{2,\dots,k+1\}$. If $\nu=k+1$, we have that $n=\lambda
  a+(a+1)-(a+1)-(k+1)a<0$, and so $n\not\in\nums$. Thus, we
  suppose that $\nu\le k$, and consequently $2\nu\le a-1$. It follows
  that $n\bmod a=(1-2\nu)\bmod a=a+1-2\nu$, and $2\lfloor n/a\rfloor =
  2\lfloor \lambda +1 -\nu +(1-2\nu)/a\rfloor=2(\lambda-\nu)$. Now,
  $a+1-2\nu\le 2\lambda-2\nu$ if and only if $a+1\le 2\lambda$, which
  does not hold, obtaining by \eqref{eq:membership} that $n\not\in
 \nums$.
\end{itemize}
This proves the two inclusions described above and concludes the proof.
\end{proof}

Observe that with the help of Lemmas~\ref{lem:ap-a-even} and
\ref{lem:ap-a-odd} and Theorem~\ref{th-appendix}, we have a full
description of $\ulf(\nums)$.  We see next what are all the
factorizations of these elements and what is their length (compare
with Theorem~\ref{theo-length}).

\begin{proposition}\label{pr:same-length}
Let $\nums=\langle a, a+1,a+2\rangle$ with $a$ a positive integer
greater than two. Every $r\in \ulf(\nums)$ can be expressed as
$r=\alpha a +\beta (a+1)+\gamma (a+2)$ with $\beta\in\{0,1\}$ and
\[   \fac(r,\nums)=\{
    (\alpha,\beta,\gamma)+\lambda(-1,2,-1) :
\lambda\in\{0,\dots,\min\{\alpha,\gamma\}\}.
 \]
Moreover, the length of any factorization of $r$ is $\lfloor
r/a\rfloor$.
\end{proposition}
\begin{proof}
Let $r$ be an element in $\nums$ whose factorizations have all the
same length. By Theorem~\ref{th-appendix}, $r\in
\operatorname{Ap}(\nums, \ubetti(\nums))$. By
Lemmas~\ref{lem:ap-a-even} and \ref{lem:ap-a-odd} we know that $r$ can
be expressed as $r=\alpha a+\beta (a+1)+\gamma(a+2)$, with
$\alpha,\gamma\in \mathbb{N}$ having some restrictions and
$\beta\in\{0,1\}$. Let $z=(\alpha,\beta,\gamma)$, which is in
$\fac(r,\nums)$. Suppose that $r$ has another factorization
$z'=(\alpha',\beta',\gamma')$. Arguing as in the proof of
Theorem~\ref{th-appendix}, we deduce that there exists
$z_1,\dots,z_n\in \fac(r,\nums)$ such that $z_1=z$, $z_n=z'$, and
$(z_i,z_{i+1})=(x_i+y_i,x_{i+1}+y_i)$ with $y_i\in\mathbb{N}^3$ and
$x_i,x_{i+1}\in \fac(b_i,\nums)$ for some $b_i\in\betti(\nums)$ (with
$x_i\neq x_{i+1}$, since the pair $(x_i,x_{i+1})$ or its symmetry is
part of a minimal presentation). Recall that, in particular, this
implies that $r\in b_i+\nums$, and as $r\in \operatorname{Ap}(\nums,
\ubetti(\nums))$, we deduce that $b_i=2(a+1)$ for all $i\in
\{1,\dots,n-1\}$. Observe that
$\fac(2(a+1),\nums)=\{(1,0,1),(0,2,0)\}$. It follows that
$z_i-z_{i+1}=x_i-x_{i+1}\in \{(1,-2,1),(-1,2,-1)\}$. Hence,
$z'-z=(z_n-z_{n-1})+\dots+(z_3-z_2)+(z_2-z_1)=\lambda (1,-2,1)$ with
$\lambda\in\mathbb{Z}$. Thus, $z'=(\alpha,\beta,\gamma)+\lambda
(-1,2,-1)$. If $\lambda<0$, from $\beta\in\{0,1\}$, we deduce that
$\beta'<0$, a contradiction. If $\lambda>\alpha$, then $\alpha'<0$,
which is impossible, whence $0\le \lambda\le \alpha$; a similar
argument shows that $\lambda\le \gamma$.

Finally, let $z=(\alpha,\beta,\gamma)+\lambda(1,-2,1)$ be a
factorization of $r$. Then, $|z|=\alpha+\beta+\gamma$. Notice that
$r=\alpha a+\beta (a+1)+\gamma a(a+2)$ and consequently $\lfloor
r/a\rfloor=\lfloor \alpha+\beta+\gamma + (\beta+2\gamma)/a\rfloor$. If
$a$ is even, then $\gamma\le a/2-1$ (Lemma~\ref{lem:ap-a-even}), and
consequently $\beta+2\gamma\le 1+a-2=a-1<a$. If $a$ is odd, then
$\gamma\le (a-1)/2$, if $\beta=0$, and $2\gamma\le (a-1)/2-1$, if
$\beta=1$. Thus, for $a$ odd, $\beta+2\gamma\le a-1<a$. In both cases,
$\beta+2\gamma<a$, yielding $\lfloor
r/a\rfloor=\alpha+\beta+\gamma=|z|$.
\end{proof}

For every $\ell\in \mathbb{N}$, recall that $\nums^\ell= \{r\in\nums :
\setl(r,\nums)=\{\ell\}\}$, that is, the set of elements in $\nums$
such that all its factorizations have length $\ell$. In view of
Theorem~\ref{th-appendix},
\[
\bigcup_{\ell\in \mathbb{N}}\nums^\ell =
\operatorname{Ap}(\nums,\ubetti(\nums)).
\]
Clearly, for $\ell\neq \ell'$, $\nums^{\ell}\cap
\nums^{\ell'}=\emptyset$. Also, according to
Lemmas~\ref{lem:ap-a-even} and \ref{lem:ap-a-odd} and
Proposition~\ref{pr:same-length}, if $\ell>a$, then
$\nums^\ell=\emptyset$. In particular,
\[
\bigcup\nolimits_{\ell=0}^{a}\nums^\ell =
\operatorname{Ap}(\nums,\ubetti(\nums))
\]
is a partition of
$\operatorname{Ap}(\nums,\ubetti(\nums))=\ulf(\nums)$. The explicit
description of $\nums^\ell$ can be derived from
Lemmas~\ref{lem:ap-a-even} and \ref{lem:ap-a-odd} (see also
Corollary~\ref{cor-length}).

\begin{proposition}\label{pr:partitionSell}
Let $\nums=\langle a, a+1,a+2\rangle$ with $a$ a positive integer
greater than two. Let $\ell\in\{0,\dots,a\}$. Then, if $a$ is even,
\[
\nums^\ell=\begin{cases} \llbracket \ell a,\ell(a+2)\rrbracket,
\text{if }0 \le \ell \le \frac{a}2-1,\\ \left\llbracket \frac{a}2 a,
(a+1)+(\frac{a}2-1)(a+2)\right\rrbracket, \text{if }
\ell=\frac{a}2,\\ \left\llbracket \frac{a}2
a+(a+1)+(\ell-\frac{a}2-1)(a+2),
(\ell-\frac{a}2)a+(a+1)+(\frac{a}2-1)(a+2) \right\rrbracket, \text{if
} \frac{a}2< \ell \le a;
    \end{cases}
\]
 while if $a$ is odd,
\[
   \nums^\ell=\begin{cases} \llbracket\ell a,
    \ell(a+2)\rrbracket, \text{if } 0\le \ell\le
    \frac{a-1}2,\\ \left\llbracket (\frac{a-1}2+1)a,
    a+\frac{a-1}2(a+2)\right\rrbracket, \text{if }
    \ell=\frac{a-1}2+1,\\ \left\llbracket
    (\frac{a-1}2+1)a+(\ell-\frac{a-1}2-1)(a+2),
    (\ell-\frac{a-1}2+1)a+(\frac{a}2-1)(a+2) \right\rrbracket,
    \text{if } \frac{a}2< \ell \le a.
    \end{cases}
    \]
\end{proposition}
\begin{proof}
Recall that by Theorem~\ref{th-appendix}, the elements having all its
factorizations of equal length are the elements of
$A:=\operatorname{Ap}(\nums,\ubetti(\nums))$,
and that Lemmas~\ref{lem:ap-a-even} and \ref{lem:ap-a-odd} describe
this set for $a$ even and odd, respectively.

First, suppose that $a$ is even. The elements in $\nums^\ell$
are the elements having an expression of the form $\lambda a + \mu
(a+1)+\eta(a+2)$ with $\lambda\in \{0,\dots, a/2\}$,
$\mu\in\{0,1\}$, $\eta\in \{0,\dots,a/2-1\}$, and
$\lambda+\mu+\eta=\ell$.
\begin{itemize}
\item If $0\le \ell \le a/2-1$, then $\ell a$ is the smallest element
  in $A$ having a factorization of length $\ell$, while $\ell(a+2)$ is
  the largest. Notice that $\ell a+2k= (\ell-k)a+k(a+2)\in A$ and
  $\ell a+2k+1=(\ell-k-1)a+(a+1)+k(a+2)\in A$, for all
  $k\in\{0,\dots,\ell-1\}$. Thus, $\nums^\ell = \llbracket \ell a,
  \ell (a+2)\rrbracket$.
\item If $\ell=\frac{a}2$, the fact that $\eta\le \frac{a}2-1$, forces
  the largest element in $\nums^\ell$ to be
  $(a+1)+(\frac{a}2-1)(a+2)$, while the smallest is $\frac{a}2a$. As
  $\frac{a}2 a+2k= (\frac{a}2 -k)a+k(a+2)$ and $\frac{a}2 a+2k+1=
  (\frac{a}2 -k-1)a+(a+1)+k(a+2)$, we deduce that
  $\nums^\ell=\llbracket
  \frac{a}2a,(a+1)+(\frac{a}2-1)(a+2)\rrbracket$.
\item Finally, let $\frac{a}2 < \ell \le a$. The smallest element in
  $\nums^\ell$ is $\frac{a}2 a+(a+1)+(\ell -\frac{a}2-1)(a+2)$, while
  the largest is $(\ell-\frac{a}2)a+(a+1)+(\frac{a}2-1)(a+2)$, and the
  rest of the argument goes as in the preceding cases.
\end{itemize}
The case $a$ odd is analogous and for this reason we leave the details
to the reader.
\end{proof}

\begin{example}
By using again \texttt{numericalsgps}, we can easily compute the sets
$\nums^\ell$.
\begin{verbatim}
gap> s:=NumericalSemigroup(10,11,12);;
gap> ap:=AperyList(s,60);;
gap> Display(List([0..10], l->Filtered(ap, 
          x->LengthsOfFactorizationsElementWRTNumericalSemigroup(x,s)=[l])));
[ [    0 ],
  [   10,   11,   12 ],
  [   20,   21,   22,   23,   24 ],
  [   30,   31,   32,   33,   34,   35,   36 ],
  [   40,   41,   42,   43,   44,   45,   46,   47,   48 ],
  [   50,   51,   52,   53,   54,   55,   56,   57,   58,   59 ],
  [   61,   62,   63,   64,   65,   66,   67,   68,   69 ],
  [   73,   74,   75,   76,   77,   78,   79 ],
  [   85,   86,   87,   88,   89 ],
  [   97,   98,   99 ],
  [  109 ] ]
\end{verbatim} 
\end{example}

Now, let $d\in \mathbb{N}$. Recall that $\nums_d=\{ r\in\nums :
\card(\fac(r,\nums))=d\}$.  Notice that from
Proposition~\ref{pr:same-length} and Lemmas~\ref{lem:ap-a-even} and
\ref{lem:ap-a-odd}, we have that the maximal denumerant of an element
with all its factorizations of the same length is $a/2$ if $a$ is
even, and $(a-1)/2+1=(a+1)/2$ if $a$ is odd. That means that
$\nums_d\cap \ulf(\nums)$ is empty for $d$ bigger than $\lceil
a/2\rceil$, and clearly
\[
\bigcup\nolimits_{d=0}^{\lceil a/2\rceil} (\nums_d\cap \ulf(\nums)) =
\operatorname{Ap}(\nums,\ubetti(\nums))
\]
is a partition of
$\operatorname{Ap}(\nums,\ubetti(\nums))=\ulf(\nums)$. Next result
related to Corollary~\ref{cor-dim}.

\begin{proposition}\label{prop:Sd}
Let $\nums=\langle a, a+1,a+2\rangle$ with $a$ a positive integer
greater than two. If $a$ is even, then for $d\in
\{1,\dots,a/2\}$,
\begin{multline*} 
   \nums_d\cap \ulf(\nums)= \left\{(d-1)a +\mu(a+1)+\eta (a+2) :
   \mu\in \{0,1\}, \eta\in
   \left\{d-1,\dots,\frac{a}2-1\right\}\right\}\\ \cup \left\{ \lambda
   a +\mu(a+1)+(d-1)(a+2) : \mu\in \{0,1\}, \lambda\in
   \left\{d,\dots,\frac{a}2\right\}\right\}.
    \end{multline*}
    If $a$ is odd, then for $d\in \{1,\ldots,(a-1)/2\}$,
    \begin{multline*} 
   \nums_d\cap \ulf(\nums)= \left\{(d-1)a +\mu(a+1)+\eta (a+2) :
   \mu\in \{0,1\}, \eta\in
   \left\{d-1,\dots,\frac{a-1}2-\mu\right\}\right\}\\ \cup \left\{
   \lambda a +\mu(a+1)+(d-1)(a+2) : \mu\in \{0,1\}, \lambda\in
   \left\{d,\dots,\frac{a-1}2+1-\mu\right\}\right\}.
    \end{multline*}    
    while for $d=(a+1)/2$,
    \[
    \nums_{\frac{a+1}2}\cap\ulf(\nums)=\left\{\frac{a-1}{2} a+\frac{a-1}{2}(a+2),
    \frac{a+1}{2} a+\frac{a-1}{2}(a+2)\right\}.
    \]
\end{proposition}
\begin{proof}
If $a$ is even, the proof follows from Theorem~\ref{th-appendix},
Proposition~\ref{pr:same-length} and Lemma~\ref{lem:ap-a-even} (in
this setting $d\in \{1,\dots,a/2\}$).  If $a$ is odd, we use
Lemma~\ref{lem:ap-a-odd} instead. For $d\in \{1,\dots,(a-1)/2\}$, the
description of $\nums_d$ is similar to the one given for $a$ even,
while $d=(a+1)/2$, forces $\eta\ge \min\{\lambda,\eta\}=d-1=(a-1)/2$,
which in turn implies that the only possible choices of $\mu$, $\eta$
and $\lambda$ in in Lemma~\ref{lem:ap-a-odd} are $\mu=0$,
$\eta=(a-1)/2$ and $\lambda\in \{(a-1)/2,(a+1)/2\}$.
\end{proof}

\begin{example}
Let us illustrate how can we use \texttt{numericalsgps} to calculate
the set $\nums_d\cap\ulf(\nums)$.
\begin{verbatim}
gap> s:=NumericalSemigroup(10,11,12);;
gap> bt:=BettiElements(s);
[ 22, 60 ]
gap> ap:=AperyList(s,bt[2]);;
gap> Display(List([1..5], l->Filtered(ap, x->Length(Factorizations(x,s))=l)));
[ [ 0, 61, 10, 11, 12, 20, 21, 23, 24, 30, 31, 35, 36, 40, 41, 47, 48, 50, 
    51, 59 ], 
  [ 62, 63, 69, 73, 22, 32, 33, 34, 42, 43, 45, 46, 52, 53, 57, 58 ], 
  [ 64, 65, 67, 68, 74, 75, 79, 85, 44, 54, 55, 56 ], 
  [ 66, 76, 77, 78, 86, 87, 89, 97 ], 
  [ 88, 98, 99, 109 ] ]
gap> s:=NumericalSemigroup(9,10,11);;
gap> bt:=BettiElements(s);
[ 20, 54, 55 ]
gap> ap:=Intersection(AperyList(s,bt[2]),AperyList(s,bt[3]));;
gap> Display(List([1..5], l->Filtered(ap, x->Length(Factorizations(x,s))=l)));
[ [   0,   9,  10,  11,  18,  19,  21,  22,  27,  28,  32,  33,  36,  37,  43,
      44,  45,  46 ],
  [  20,  29,  30,  31,  38,  39,  41,  42,  47,  48,  52,  53,  56,  57 ],
  [  40,  49,  50,  51,  58,  59,  61,  62,  67,  68 ],
  [  60,  69,  70,  71,  78,  79 ],
  [  80,  89 ] ]
\end{verbatim}    
\end{example}

\section*{Acknowledgements}

We are thankful to Francesc Aguil\'o-Gost, Maria Bras-Amor\'os,
Ignacio Ojeda and J. Carlos Rosales for their interest and for helping
us to contextualise our work and providing us with useful and
interesting references.

\end{document}